\documentclass{amsart}
\usepackage{amssymb}
\usepackage{amsfonts}

\newtheorem{theorem}{Theorem}[section]

\newtheorem{corollary}[theorem]{Corollary}
\newtheorem{definition}[theorem]{Definition}
\newtheorem{example}[theorem]{Example}
\newtheorem{lemma}[theorem]{Lemma}
\newtheorem{proposition}[theorem]{Proposition}
\theoremstyle{remark}
\newtheorem{remark}[theorem]{Remark}
\numberwithin{equation}{section}
\input{tcilatex}

\begin{document}
\title{A STUDY OF THREE-DIMENSIONAL PARACONTACT $(\tilde{\kappa},\tilde{\mu},%
\tilde{\nu})$-SPACES}
\author{I. KUPELI ERKEN}
\address{Faculty of Natural Sciences, Architecture and Engineering,
Department of Mathematics, Bursa Technical University, Bursa, TURKEY}
\email{irem.erken@btu.edu.tr}
\author{C. MURATHAN}
\address{Art and Science Faculty, Department of Mathematics, Uludag
University, 16059 Bursa, TURKEY}
\email{cengiz@uludag.edu.tr}
\keywords{Paracontact metric manifold, para-Sasakian, contact metric
manifold, $(\kappa ,\mu )$-manifold.\\
This study is a part of PhD thesis of the first author.}

\begin{abstract}
This paper is a study of three-dimensional paracontact metric $(\tilde{\kappa%
},\tilde{\mu},\tilde{\nu})$-manifolds. Three dimensional paracontact metric
manifolds whose Reeb vector field $\xi $ is harmonic are characterized. We
focus on some curvature properties by considering the class of paracontact
metric $(\tilde{\kappa},\tilde{\mu},\tilde{\nu})$-manifolds under a
condition which is given at Definition 3.1. We study properties of such
manifolds according to the cases $\tilde{\kappa}>-1,$ $\tilde{\kappa}=-1,~%
\tilde{\kappa}<-1$ and construct new examples of such manifolds for each
case. We also show that the existence of paracontact metric $(-1,\tilde{\mu}%
\neq 0,\tilde{\nu}\neq 0)$ spaces with dimension greater than $3$ such that $%
\tilde{h}^{2}=0$ but $\tilde{h}\neq 0.$
\end{abstract}

\maketitle

\section{I\textbf{ntroduction}}

\label{introduction}

Paracontact manifolds are smooth manifolds of dimension $2n+1$ endowed with
a $1$-form $\eta $, a vector field $\xi $ and a field of endomorphisms of
tangent spaces $\tilde{\varphi}$ such that $\eta (\xi )=1$, $\tilde{\varphi}%
^{2}=I-\eta \otimes \xi $ and $\tilde{\varphi}$ induces an almost
paracomplex structure on the codimension $1$ distribution defined by the
kernel of $\eta $ (see $\S $ \ref{preliminaries} for more details). In
addition, if the manifold is endowed with a pseudo-Riemannian metric $\tilde{%
g}$ of signature $(n+1,n)$ satisfying 
\begin{equation*}
\tilde{g}(\tilde{\varphi}X,\tilde{\varphi}Y)=-\tilde{g}(X,Y)+\eta (X)\eta
(Y),\ \ \ d\eta (X,Y)=\tilde{g}(X,\tilde{\varphi}Y),
\end{equation*}%
$(M,\eta )$ becomes a contact manifold and $(\tilde{\varphi},\xi ,\eta ,%
\tilde{g})$ is said to be a paracontact metric structure on $M$. \ The study
of paracontact geometry was started by Kaneyuki and Williams in \cite%
{kaneyuki1} and then it was continued by many other authors. A systematic
study of paracontact metric manifolds, and some subclasses like
para-Sasakian manifolds, was carried out in a striking paper of Zamkovoy 
\cite{Za}. The importance of paracontact geometry, and in particular of
para-Sasakian geometry, has been pointed out especially in the last years by
several papers highlighting the interplays with the theory of para-K\"{a}%
hler manifolds and its role in pseudo-Riemannian geometry and mathematical
physics (cf. e.g. \cite{alekseevski1}, \cite{alekseevski2}, \cite{CMEM}, 
\cite{cortes1}, \cite{cortes2}).

In \cite{CDFK}, Closset, Dumitrescu, Festuccia, and Komargodski constructed
supersymmetric field theories on Riemannian three manifolds, considering
three dimensional theories with $\mathcal{N}$ $=2$ supersymmetry. They
proved that the supersymmetric field theory on three manifold $M$ possesses
a single supercharge if and only if $M$ admits an almost contact metric
structure that satisfies certain integrability conditions. Recently, Willett 
\cite{Willett} studied the localization of three dimensional $\mathcal{N}$ $%
=2$ supersymmetric theories on compact manifolds.

As known every orientable Riemannian three-manifold admits a
metric-compatible almost contact structure and three dimensional unit sphere 
$\emph{S}^{3}$ has a Sasakian structure. In \cite{MK} Markellos and
Tsichlias constructed new $(\kappa ,\mu )$ contact metric structures
(non-Sasakian) on the unit sphere $\emph{S}^{3}$.

On the other hand, in \cite{MOTE} it was proved (cf. Theorem \ref{motivation}
below) that any (non-Sasakian) contact $(\kappa ,\mu )$-space carries a
canonical paracontact metric structure $(\tilde{\varphi},\xi ,\eta ,\tilde{g}%
)$ whose Levi-Civita connection satisfies a condition formally similar to
contact case 
\begin{equation}
\tilde{R}(X,Y)\xi =\tilde{\kappa}\left( \eta \left( Y\right) X-\eta \left(
X\right) Y\right) +\tilde{\mu}(\eta \left( Y\right) \tilde{h}X-\eta \left(
X\right) \tilde{h}Y),  \label{paranullity}
\end{equation}%
where $2\tilde{h}:={\mathcal{L}}_{\xi }\tilde{\varphi}$ and, in this case, $%
\tilde{\kappa}=(1-\mu /2)^{2}+\kappa -2$, $\tilde{\mu}=2$. \ By \cite{MOTE}
and \cite{MK}, $\emph{S}^{3}$ will have paracontact metric structure.

A $(2n+1)$-dimensional paracontact metric manifold $(M,\tilde{\varphi},\xi
,\eta ,\tilde{g})$ whose curvature tensor satisfies (\ref{paranullity}), is
called \emph{paracontact $(\tilde{\kappa},\tilde{\mu})$-manifold}. The class
of paracontact $(\tilde{\kappa},\tilde{\mu})$-manifolds is very large. It
contains para-Sasakian manifolds, as well as those paracontact metric
manifolds satisfying $\tilde{R}(X,Y)\xi =0$ for all $X,Y\in \Gamma (TM)$
(recently studied in \cite{ZATZA}). But, unlike in the contact Riemannian
case, a paracontact $(\tilde{\kappa},\tilde{\mu})$-manifold such that $%
\tilde{\kappa}=-1$ in general is not para-Sasakian. In fact, there are
paracontact $(\tilde{\kappa},\tilde{\mu})$-manifolds such that $\tilde{h}%
^{2}=0$ (which is equivalent to take $\tilde{\kappa}=-1$) but with $\tilde{h}%
\neq 0$. Another important difference with the contact Riemannian case, due
to the non-positive definiteness of the metric, is that while for contact
metric $(\kappa ,\mu )$-spaces the constant $\kappa $ can not be greater
than $1$, here we have no restriction for the constants $\tilde{\kappa}$ and 
$\tilde{\mu}$. It should be also remarked that contact metric $(\kappa ,\mu
,\nu )$-spaces of dimension greater than $3$ is either a Sasakian manifold
or a $(\kappa ,\mu )$-contact metric manifold. But the authors provided an
example of paracontact metric $(\tilde{\kappa},\tilde{\mu},\tilde{\upsilon})$%
-manifold such that $\tilde{\kappa}=-1,($and $\tilde{h}^{2}=0)$ but with $%
\tilde{\mu}\neq 0,\tilde{\nu}\neq 0,$ $\tilde{h}\neq 0$ , and $n>1.$

Cappelletti Montano et al.\cite{CMEM} showed in that there is a kind of
duality between those manifolds and contact metric $(\kappa ,\mu )$-spaces
and also proved that $\xi $ is a Ricci eigenvector of paracontact $(\tilde{%
\kappa},\tilde{\mu})$-manifolds.

In \cite{CAPER}, G.Calvaruso and D. Perrone proved that $\xi $ is \textit{%
harmonic} if and only if $\xi $ is an eigenvector of the Ricci operator for
contact semi-Riemannian manifolds.

It turns out that there exists a motivation to study harmonic maps in
contact semi-Riemannian and paracontact geometry.

Let $(M,g)$ be smooth, oriented, connected pseudo-Riemannian manifold and $%
(TM,g^{S})$ its tangent bundle endowed with the Sasaki metric (also referred
to as Kaluza-Klein metric in Mathematical Physics) $g^{S}$. By definition,
the \textit{energy} of a smooth vector field $V$ on $M$ is the energy
corresponding $V:(M,g)\rightarrow (TM,g^{s}).$ When $M$ is compact, the 
\textit{energy} of $V$ is determined by%
\begin{equation*}
E(V)=\frac{1}{2}\int_{M}(tr_{g}V^{\ast }g^{s})dv=\frac{n}{2}vol(M,g)+\frac{1%
}{2}\int_{M}\left\Vert \nabla V\right\Vert ^{2}dv.
\end{equation*}
The non-compact case, one can take into account over relatively compact
domains. It can be shown that $V:(M,g)\rightarrow (TM,g^{s})$ is harmonic
map if and only if%
\begin{equation}
tr\left[ R(\nabla .V,V).\right] =0,\ \nabla ^{\ast }\nabla V=0,
\label{NAMBLASTAR2}
\end{equation}%
where 
\begin{equation}
\nabla ^{\ast }\nabla V=\sum_{i}\varepsilon _{i}(\nabla _{e_{i}}\nabla
_{e_{i}}V-\nabla _{\nabla _{e_{i}}e_{i}}V)  \label{NAMBLASTAR}
\end{equation}%
is the rough Laplacian with respect to a pseudo-orthonormal local frame $%
\left\{ e_{1},...,e_{n}\right\} $ on $(M,g)$ with $g(e_{i},e_{i})=%
\varepsilon _{i}=\pm 1$ for all indices $i=1,...,n.$

If ($M,g)$ is a compact Riemannian manifold, only parallel vector fields
define harmonic maps.

Next, for any real constant $\rho \neq 0$, let $\chi ^{\rho }(M)=\left\{
W\in \chi (M):\left\Vert W\right\Vert ^{2}=\rho \right\} .$We consider
vector fields $V\in $ $\chi ^{\rho }(M)$ which are critical points for the
energy functional $E\mid _{\chi ^{\rho }(M)}$, restricted to vector fields
of the same length. The Euler-Lagrange equations of this variational
condition yield that $V$ is a harmonic vector field if and only if%
\begin{equation}
\nabla ^{\ast }\nabla V\text{ is collinear to }V.  \label{NAMBLASTAR4}
\end{equation}

This characterization is well known in the Riemannian case ([3, 23, 25]).
Using same arguments in pseudo-Riemannian case, G. Calvaruso \cite%
{calvaruso1} proved that same result is still valid for vector fields of
constant length, if it is not lightlike.

Let $T_{1}M$ denote the unit tangent sphere bundle over $M$, and again by $%
g^{S}$ the metric induced on $T_{1}M$ by the Sasaki metric of $TM.$ Then, it
is shown that in \cite{ACP}, the map on $V:(M,g)\rightarrow (T_{1}M,g^{s})$
is harmonic if $V$ is a harmonic vector field and the additonial condition 
\begin{equation}
tr[R(\nabla .V,V).]=0  \label{TM}
\end{equation}%
is satisfied. G. Calvaruso \cite{calvaruso1} also investigated harmonicity
properties for left-invariant vector fields on three-dimensional Lorentzian
Lie groups, obtaining several classification results and new examples of
critical points of energy functionals. In \cite{calvaruso2}, he studied
harmonicity properties of vector fields on four-dimensional
pseudo-Riemannian generalized symmetric spaces. Moreover, he gave a complete
classification of three-dimensional homogeneous paracontact metric manifolds
in \cite{calvar}. Recently, G.Calvaruso and D.Perrone \cite{CAPERR} proved
that all three-dimensional homogeneous paracontact metric manifolds are 
\emph{H}\textit{-paracontact}, that is, paracontact metric manifolds whose
characteristic vector field $\xi $ is harmonic.

In this paper, we study harmonicity of the characteristic vector field of
three-dimensional paracontact metric manifolds and give a characterization
of $(\tilde{\kappa},\tilde{\mu},\tilde{\nu})$\textit{-}manifolds\textit{. }

\textbf{\textit{Overview:}} Here is the plan of the paper: $\S $ \ref%
{preliminaries} is devoted to preliminaries. In $\S $ \ref{third}, we study
the common properties of paracontact metric $(\tilde{\kappa},\tilde{\mu},%
\tilde{\nu})$-manifolds (see $\S $ \ref{third} for definition) for the cases 
$\tilde{\kappa}<-1$, $\tilde{\kappa}=-1$, $\tilde{\kappa}>-1.$ Beside the
other results, we prove for instance that while the values of $\tilde{\kappa}%
,$ $\tilde{\mu}$ and $\tilde{\nu}$ change, the paracontact metric $(\tilde{%
\kappa},\tilde{\mu},\tilde{\upsilon})$-manifolds remain unchanged under $%
\mathcal{D}$-homothetic deformations. Moreover we prove that in $\dim $ $M$ $%
>3$, paracontact metric $(\tilde{\kappa}\neq -1,\tilde{\mu},\tilde{\nu})$%
-manifolds must be paracontact $(\tilde{\kappa},\tilde{\mu})$-manifolds.

We reserve $\S $ \ref{fourth}, for the following two main theorems of the
paper:

\begin{theorem}
\label{DESTAR}Let $(M,\tilde{\varphi},\xi ,\eta ,\tilde{g})$ be a $3$%
-dimensional paracontact metric manifold. $\xi $ is a harmonic vector field
if and only if the characteristic vector field $\xi $ is an eigenvector of
the Ricci operator.
\end{theorem}

\begin{theorem}
\label{k mu vu}\textit{Let }$(M,\tilde{\varphi},\xi ,\eta ,\tilde{g})$%
\textit{\ be a }$3$\textit{-dimensional paracontact metric manifold. If \
the characteristic vector field }$\xi $\textit{\ is harmonic vector field
then the paracontact metric }$(\tilde{\kappa},\tilde{\mu},\tilde{\nu})$%
\textit{-manifold always exists on every open and dense subset of }$M.$%
\textit{\ Conversely, if }$M$\textit{\ is a paracontact metric }$(\tilde{%
\kappa},\tilde{\mu},\tilde{\nu})$\textit{-manifold then the characteristic
vector field }$\xi $\textit{\ is harmonic vector field.}
\end{theorem}

We also show that a paracontact metric $(\tilde{\kappa},\tilde{\mu},\tilde{%
\nu})$-manifold with $\tilde{\kappa}=-1$ is not necessary para-Sasakian (see
the tensor $\tilde{h}$\ has the canonical form (II)). As stated above, this
case shows important difference with the contact Riemannian case. So we can
construct non-trival examples of non-para-Sasakian manifold with $\tilde{%
\kappa}=-1$. Also one could find examples about paracontact metric $(\tilde{%
\kappa},\tilde{\mu},\tilde{\nu})$-manifolds according to the cases $\tilde{%
\kappa}>-1,$ $~\tilde{\kappa}<-1.$

In $\S $ \ref{fifth}, we give a relation between non-Sasakian $(\kappa ,%
\mu
,\nu =const.)$-contact metric manifold with the Boeckx invariant $I_{M}=%
\frac{1-\frac{\mu }{2}}{\sqrt{1-\kappa }}$ is constant along the integral
curves of $\xi $ i.e. $\xi (I_{M})=0$ and $3$-dimensional paracontact metric 
$(\tilde{\kappa},\tilde{\mu},\tilde{\nu})$-manifold.

\section{Preliminaries}

\label{preliminaries}

\medskip An $(2n+1)$-dimensional smooth manifold $M$ is said to have an 
\emph{almost paracontact structure} if it admits a $(1,1)$-tensor field $%
\tilde{\varphi}$, a vector field $\xi $ and a $1$-form $\eta $ satisfying
the following conditions:

\begin{enumerate}
\item[(i)] $\eta(\xi )=1$, \ $\tilde{\varphi}^{2}=I-\eta \otimes \xi$,

\item[(ii)] the tensor field $\tilde{\varphi}$ induces an almost paracomplex
structure on each fibre of ${\mathcal{D}}=\ker(\eta)$, i.e. the $\pm 1$%
-eigendistributions, ${\mathcal{D}}^{\pm}:={\mathcal{D}}_{\tilde\varphi}(\pm
1)$ of $\tilde\varphi$ have equal dimension $n$.
\end{enumerate}

From the definition it follows that $\tilde{\varphi}\xi =0$, $\eta \circ 
\tilde{\varphi}=0$ and the endomorphism $\tilde{\varphi}$ has rank $2n$.
When the tensor field $N_{\tilde{\varphi}}:=[\tilde{\varphi},\tilde{\varphi}%
]-2d\eta \otimes \xi $ vanishes identically the almost paracontact manifold
is said to be \emph{normal}. If an almost paracontact manifold admits a
pseudo-Riemannian metric $\tilde{g}$ such that 
\begin{equation}
\tilde{g}(\tilde{\varphi}X,\tilde{\varphi}Y)=-\tilde{g}(X,Y)+\eta (X)\eta
(Y),  \label{G METRIC}
\end{equation}%
for all $X,Y\in \Gamma (TM)$, then we say that $(M,\tilde{\varphi},\xi ,\eta
,\tilde{g})$ is an \textit{almost paracontact metric manifold}. Notice that
any such a pseudo-Riemannian metric is necessarily of signature $(n+1,n)$.
For an almost paracontact metric manifold, there always exists an orthogonal
basis $\{X_{1},\ldots ,X_{n},Y_{1},\ldots ,Y_{n},\xi \}$ such that $\tilde{g}%
(X_{i},X_{j})=\delta _{ij}$, $\tilde{g}(Y_{i},Y_{j})=-\delta _{ij}$ and $%
Y_{i}=\tilde{\varphi}X_{i}$, for any $i,j\in \left\{ 1,\ldots ,n\right\} $.
Such basis is called a $\tilde{\varphi}$-basis.

If in addition $d\eta (X,Y)=\tilde{g}(X,\tilde{\varphi}Y)$ for all vector
fields $X,Y$ on $M,$ $(M,\tilde{\varphi},\xi ,\eta ,\tilde{g})$ is said to
be a \emph{paracontact metric manifold}. In a paracontact metric manifold
one defines a symmetric, trace-free operator $\tilde{h}:=\frac{1}{2}{%
\mathcal{L}}_{\xi }\tilde{\varphi}$. It is known \cite{Za} that $\tilde{h}$
anti-commutes with $\tilde{\varphi}$ and satisfies tr$\tilde{h}=0=$ $\tilde{h%
}\xi $ and 
\begin{equation}
\tilde{\nabla}\xi =-\tilde{\varphi}+\tilde{\varphi}\tilde{h},
\label{nablaxi}
\end{equation}%
where $\tilde{\nabla}$ is the Levi-Civita connection of the
pseudo-Riemannian manifold $(M,\tilde{g})$. Moreover $\tilde{h}\equiv 0$ if
and only if $\xi $ is a Killing vector field and in this case $(M,\tilde{%
\varphi},\xi ,\eta ,\tilde{g})$ is said to be a \emph{K-paracontact manifold}%
. A normal paracontact metric manifold is called a \textit{para-Sasakian
manifold}. Also in this context the para-Sasakian condition implies the $K$%
-paracontact condition and the converse holds only in dimension $3$ (see 
\cite{calvar}). Moreover, in any para-Sasakian manifold 
\begin{equation}
\tilde{R}(X,Y)\xi =-(\eta (Y)X-\eta (X)Y)  \label{Pasa}
\end{equation}%
holds, but unlike contact metric geometry the condition (\ref{Pasa}) not
necessarily implies that the manifold is para-Sasakian. Differentiating $%
\tilde{\nabla}_{Y}\xi =-\tilde{\varphi}Y+\tilde{\varphi}\tilde{h}Y$, we get 
\begin{equation}
\tilde{R}(X,Y)\xi =-(\tilde{\nabla}_{X}\tilde{\varphi})Y+(\tilde{\nabla}_{Y}%
\tilde{\varphi})X+(\tilde{\nabla}_{X}\tilde{\varphi}\tilde{h})Y-(\tilde{%
\nabla}_{Y}\tilde{\varphi}\tilde{h})X.  \label{CURVATURE 4}
\end{equation}%
In \cite{Za}, Zamkovoy proved that%
\begin{equation}
(\tilde{\nabla}_{\xi }\tilde{h})X=-\tilde{\varphi}X+\tilde{h}^{2}\tilde{%
\varphi}X+\tilde{\varphi}\tilde{R}(\xi ,X)\xi .  \label{irem000}
\end{equation}%
Moreover, he showed that Ricci curvature $\tilde{S}$ in the direction of $%
\xi $ is given by%
\begin{equation}
\tilde{S}(\xi ,\xi )=-2n+tr \tilde{h}^{2}.  \label{RICCI ZETA}
\end{equation}

An almost paracontact structure ($\tilde{\varphi},\xi ,\eta )$ is said to be 
\emph{integrable} if $N_{\tilde{\varphi}}(X,Y)\in \Gamma (\mathbb{R}\xi )$
whenever $X,Y\in \Gamma ({\mathcal{D}})$.

In \cite{JOANNA}, J. Welyczko proved that any $3$-dimensional paracontact
metric manifold is always integrable. So for $3$-dimensional paracontact
metric manifold, we have%
\begin{equation}
(\tilde{\nabla}_{X}\tilde{\varphi})Y=-\tilde{g}(X-\tilde{h}X,Y)\xi +\eta
(Y)(X-\tilde{h}X).  \label{INTEGRABLE}
\end{equation}

We end this section by pointing out the following.

\begin{theorem}[\protect\cite{MOTE}]
\label{motivation} Let $(M,\varphi ,\xi ,\eta ,g)~$\ be a non-Sasakian
contact metric $(\kappa ,\mu )$-space. Then $M$ admits a canonical
paracontact metric structure $(\tilde{\varphi},\xi ,\eta ,\tilde{g})$ given
by 
\begin{equation}
\tilde{\varphi}:=\frac{1}{\sqrt{1-\kappa }}h,\ \ \tilde{g}:=\frac{1}{\sqrt{%
1-\kappa }}d\eta (\cdot ,h\cdot )+\eta \otimes \eta .  \label{CAPAR1}
\end{equation}
\end{theorem}

\section{Preliminary results on 2n+1-dimensional paracontact metric $(\tilde{%
\protect\kappa},\tilde{\protect\mu},\tilde{\protect\nu})$-manifolds}

\label{third}

Theorem \ref{motivation} motivates the following definition.

\begin{definition}
\label{pkm}A $2n+1$-dimensional paracontact metric $(\tilde{\kappa},\tilde{%
\mu},\tilde{\nu})$-manifold is a paracontact metric manifold for which the
curvature tensor field satisfies 
\begin{equation}
\tilde{R}(X,Y)\xi =\tilde{\kappa}(\eta (Y)X-\eta (X)Y)+\tilde{\mu}(\eta (Y)%
\tilde{h}X-\eta (X)\tilde{h}Y)+\tilde{\nu}(\eta (Y)\tilde{\varphi}\tilde{h}%
X-\eta (X)\tilde{\varphi}\tilde{h}Y),  \label{PARAKMU}
\end{equation}%
for all $X,Y\in \Gamma (TM)$, where $\tilde{\kappa},\tilde{\mu},\tilde{\nu}$
are smooth functions on $M$.
\end{definition}

In this section, we discuss some properties of paracontact metric manifolds
satisfying the condition \eqref{PARAKMU}. We start with some preliminary
properties.

\begin{lemma}
\label{hq}Let $(M,\tilde{\varphi},\xi ,\eta ,\tilde{g})$ be a $2n+1$%
-dimensional paracontact metric $(\tilde{\kappa},\tilde{\mu},\tilde{\nu})$%
-manifold. Then the following identities hold: 
\begin{gather}
\tilde{h}^{2}=(1+\tilde{\kappa})\tilde{\varphi}^{2},  \label{H2} \\
\tilde{Q}\xi =2n\tilde{\kappa}\xi ,  \label{Riczeta} \\
(\tilde{\nabla}_{X}\tilde{\varphi})Y=-\tilde{g}(X-\tilde{h}X,Y)\xi +\eta
(Y)(X-\tilde{h}X),\ \text{for }\tilde{\kappa}\neq -1,\text{ }
\label{NAMBLAFI} \\
(\tilde{\nabla}_{X}\tilde{h})Y-(\tilde{\nabla}_{Y}\tilde{h})X=-(1+\tilde{%
\kappa})(2\tilde{g}(X,\tilde{\varphi}Y)\xi +\eta (X)\tilde{\varphi}Y-\eta (Y)%
\tilde{\varphi}X)  \label{NAMLA X H} \\
\quad \quad \quad \quad \quad +(1-\tilde{\mu})(\eta (X)\tilde{\varphi}\tilde{%
h}Y-\eta (Y)\tilde{\varphi}\tilde{h}X)  \notag \\
-\tilde{\nu}(\eta (X)\tilde{h}Y-\eta (Y)\tilde{h}X),\text{\ }  \notag \\
(\tilde{\nabla}_{X}\tilde{\varphi}\tilde{h})Y-(\tilde{\nabla}_{Y}\tilde{%
\varphi}\tilde{h})X=-(1+\tilde{\kappa})(\eta (X)Y-\eta (Y)X)
\label{NAMLAFIH} \\
+(1-\tilde{\mu})(\eta (X)\tilde{h}Y-\eta (Y)\tilde{h}X)  \notag \\
-\tilde{\nu}(\eta (X)\tilde{\varphi}\tilde{h}Y-\eta (Y)\tilde{\varphi}\tilde{%
h}X),  \notag \\
\tilde{\nabla}_{\xi }\tilde{h}=\tilde{\mu}\tilde{h}\circ \tilde{\varphi}-%
\tilde{\nu}\tilde{h},\ \ \ \tilde{\nabla}_{\xi }\tilde{\varphi}\tilde{h}=-%
\tilde{\mu}\tilde{h}+\tilde{\nu}\tilde{h}\circ \tilde{\varphi},\text{ \ }
\label{NMBLA ZETAH} \\
\tilde{R}_{\xi X}Y=\tilde{\kappa}(\tilde{g}(X,Y)\xi -\eta (Y)X)+\tilde{\mu}(%
\tilde{g}(\tilde{h}X,Y)\xi -\eta (Y)\tilde{h}X)  \label{RZETAXY} \\
+\tilde{\nu}(\tilde{g}(\tilde{\varphi}\tilde{h}X,Y)\xi -\eta (Y)\tilde{%
\varphi}\tilde{h}X),  \notag \\
\xi (\tilde{\kappa})=-2\tilde{\nu}(1+\tilde{\kappa}),  \label{ZETAK VU1}
\end{gather}%
for any vector fields $X$, $Y$ on $M$, where $\tilde{Q}$ denotes the Ricci
operator of $(M,\tilde{g})$.

\begin{proof}
The proof of (\ref{H2})-(\ref{NAMLA X H}) are similar to that of \ [\cite%
{CMEM}, Lemma 3.2]. The relation (\ref{NAMLAFIH}) is an immediate
consequence of (\ref{NAMBLAFI}), (\ref{NAMLA X H}) and $(\tilde{\nabla}_{X}%
\tilde{\varphi}\tilde{h})Y-(\tilde{\nabla}_{Y}\tilde{\varphi}\tilde{h})X=(%
\tilde{\nabla}_{X}\tilde{\varphi})\tilde{h}Y-(\tilde{\nabla}_{Y}\tilde{%
\varphi})\tilde{h}X+\tilde{\varphi}((\tilde{\nabla}_{X}\tilde{h})Y-(\tilde{%
\nabla}_{Y}\tilde{h})X).$ Putting $\xi $ instead of $X$ in the relation (\ref%
{NAMLA X H}), we have (\ref{NMBLA ZETAH}). Taking into account $\tilde{h}%
\tilde{\varphi}=-\tilde{\varphi}\tilde{h}$, (\ref{H2}) and (\ref{NMBLA ZETAH}%
), we obtain 
\begin{equation}
\tilde{\nabla}_{\xi }\tilde{h}^{2}=(\tilde{\nabla}_{\xi }\tilde{h})\tilde{h}+%
\tilde{h}(\tilde{\nabla}_{\xi }\tilde{h})=(\tilde{\mu}\tilde{h}\tilde{\varphi%
}-\tilde{\nu}\tilde{h})\tilde{h}+\tilde{h}(\tilde{\mu}\tilde{h}\tilde{\varphi%
}-\tilde{\nu}\tilde{h})=-2\tilde{\nu}(1+\tilde{\kappa})\tilde{\varphi}^{2}.
\label{zetah2}
\end{equation}%
Alternately, differentiating (\ref{H2}) along $\xi $ and using (\ref%
{INTEGRABLE}), we obtain 
\begin{equation}
\tilde{\nabla}_{\xi }\tilde{h}^{2}=\xi (\tilde{\kappa})\tilde{\varphi}^{2}.
\label{zetak}
\end{equation}

Combining (\ref{zetah2}) and (\ref{zetak}), we complete the proof (\ref%
{ZETAK VU1}).
\end{proof}
\end{lemma}

Remarkable subclasses of paracontact $(\tilde{\kappa},\tilde{\mu},\tilde{\nu}%
)$-manifolds are given, in view of \eqref{Pasa}, by para-Sasakian manifolds,
and by those paracontact metric manifolds such that $\tilde{R}(X,Y)\xi =0$
for all vector fields $X,Y$ on $M$. Note that, by \eqref{H2}, a paracontact $%
(\tilde{\kappa},\tilde{\mu},\tilde{\nu})$-manifold such that $\tilde{\kappa}%
=-1$ satisfies $\tilde{h}^{2}=0$. Unlike the contact metric case, since the
metric $\tilde{g}$ is pseudo-Riemannian we can not conclude that $\tilde{h}$
vanishes and so the manifold is para-Sasakian.

\medskip Given a paracontact metric structure $(\tilde{\varphi},\xi ,\eta ,%
\tilde{g})$ and $\alpha >0$, the change of structure tensors 
\begin{equation}
\bar{\eta}=\alpha \eta ,\text{ \ \ }\bar{\xi}=\frac{1}{\alpha }\xi ,\text{ \
\ }\bar{\varphi}=\tilde{\varphi},\text{ \ \ }\bar{g}=\alpha \tilde{g}+\alpha
(\alpha -1)\eta \otimes \eta  \label{DHOMOTHETIC}
\end{equation}%
is called a \emph{${\mathcal{D}}_{\alpha }$-homothetic deformation}. One can
easily check that the new structure $(\bar{\varphi},\bar{\xi},\bar{\eta},%
\bar{g})$ is still a paracontact metric structure \cite{Za}.

\begin{proposition}[\protect\cite{CMEM}]
\label{levicivita} Let $(\bar{\varphi},\bar{\xi},\bar{\eta},\bar{g})$ be a
paracontact metric structure obtained from $(\tilde{\varphi},\tilde{\xi},%
\tilde{\eta},\tilde{g})$ by a ${\mathcal{D}}_{\alpha }$-homothetic
deformation. Then we have the following relationship between the Levi-Civita
connections $\bar{\nabla}$ and $\tilde{\nabla}$ of $\bar{g}$ and $\tilde{g}$%
, respectively, 
\begin{equation}
\bar{\nabla}_{X}Y=\tilde{\nabla}_{X}Y+\frac{\alpha -1}{\alpha }\tilde{g}(%
\tilde{\varphi}\tilde{h}X,Y)\xi -(\alpha -1)\left( \eta (Y)\tilde{\varphi}%
X+\eta (X)\tilde{\varphi}Y\right) .  \label{CONNECTION}
\end{equation}%
Furthermore, 
\begin{equation}
\bar{h}=\frac{1}{\alpha }\tilde{h}.  \label{H BAR}
\end{equation}
\end{proposition}

After some straightforward calculations one can prove the following
proposition.

\begin{proposition}[\protect\cite{CMEM}]
\label{levicivita 2}Under the same assumptions of Proposition \ref%
{levicivita}, the curvature tensor fields $\bar{R}$ and $\tilde{R}$ are
related by \ 
\begin{align}
\alpha \bar{R}(X,Y)\bar{\xi}& =\tilde{R}(X,Y)\xi -(\alpha -1)((\tilde{\nabla}%
_{X}\tilde{\varphi})Y-(\tilde{\nabla}_{Y}\tilde{\varphi})X+\eta (Y)(X-\tilde{%
h}X)-\eta (X)(Y-\tilde{h}Y))  \notag \\
& \quad -(\alpha -1)^{2}(\eta (Y)X-\eta (X)Y).  \label{CURVATURE}
\end{align}
\end{proposition}

Using Proposition \ref{levicivita 2} one can easily get following

\begin{proposition}
\label{kmv1}If $(M,\tilde{\varphi},\xi ,\eta ,\tilde{g})$ is a paracontact
metric $(\tilde{\kappa},\tilde{\mu},\tilde{\nu})$-manifold, then $(\bar{%
\varphi},\bar{\xi},\bar{\eta},\bar{g})$ is a paracontact metric $(\bar{\kappa%
},\bar{\mu},\bar{\nu})$-structure, where 
\begin{equation}
\bar{\kappa}=\frac{\tilde{\kappa}+1-\alpha ^{2}}{\alpha ^{2}},\ \ \ \bar{\mu}%
=\frac{\tilde{\mu}+2\alpha -2}{\alpha },\text{ \ }\bar{\nu}=\frac{\tilde{\nu}%
}{\alpha }\text{\ }.  \label{dhomothetic1}
\end{equation}
\end{proposition}

By the same argument in Proposition 3.9 of \cite{CMEM}, we have

\begin{corollary}
\label{orthogonal} \label{bipara2} Let $(M,\tilde{\varphi},\xi ,\eta ,\tilde{%
g})$ be a paracontact metric $(\tilde{\kappa},\tilde{\mu},\tilde{\nu})$%
-manifold such that $\tilde{\kappa}\neq -1$. Then the operator $\tilde{h}$
in the case $\tilde{\kappa}>-1$ and the operator $\tilde{\varphi}\tilde{h}$
in the case $\tilde{\kappa}<-1$ are diagonalizable and admit three
eigenvalues: $0$, associate with the eigenvector $\xi $, $\tilde{\lambda}$
and -$\tilde{\lambda}$, of multiplicity $n$, where $\tilde{\lambda}:=\sqrt{%
|1+\tilde{\kappa}|}$. The corresponding eigendistributions ${\mathcal{D}}_{%
\tilde{h}}(0)=\mathbb{R}\xi $, ${\mathcal{D}}_{\tilde{h}}(\tilde{\lambda})$, 
${\mathcal{D}}_{\tilde{h}}(-\tilde{\lambda})$ and ${\mathcal{D}}_{\tilde{%
\varphi}\tilde{h}}(0)=\mathbb{R}\xi $, ${\mathcal{D}}_{\tilde{\varphi}\tilde{%
h}}(\tilde{\lambda})$, ${\mathcal{D}}_{\tilde{\varphi}\tilde{h}}(-\tilde{%
\lambda})$ are mutually orthogonal and one has $\tilde{\varphi}{\mathcal{D}}%
_{\tilde{h}}(\tilde{\lambda})={\mathcal{D}}_{\tilde{h}}(-\tilde{\lambda})$, $%
\tilde{\varphi}{\mathcal{D}}_{\tilde{h}}(-\tilde{\lambda})={\mathcal{D}}_{%
\tilde{h}}(\tilde{\lambda})$ and $\tilde{\varphi}{\mathcal{D}}_{\tilde{%
\varphi}\tilde{h}}(\tilde{\lambda})={\mathcal{D}}_{\tilde{\varphi}\tilde{h}%
}(-\tilde{\lambda})$, $\tilde{\varphi}{\mathcal{D}}_{\tilde{\varphi}\tilde{h}%
}(-\tilde{\lambda})={\mathcal{D}}_{\tilde{\varphi}\tilde{h}}(\tilde{\lambda}%
) $. Furthermore, 
\begin{equation}
{\mathcal{D}}_{\tilde{h}}(\pm \tilde{\lambda})=\left\{ X\pm \frac{1}{\sqrt{1+%
\tilde{\kappa}}}\tilde{h}X|X\in \Gamma ({\mathcal{D}}^{\mp })\right\}
\label{formula1}
\end{equation}%
in the case $\tilde{\kappa}>-1$, and 
\begin{equation}
{\mathcal{D}}_{\tilde{\varphi}\tilde{h}}(\pm \tilde{\lambda})=\left\{ X\pm 
\frac{1}{\sqrt{-1-\tilde{\kappa}}}\tilde{\varphi}\tilde{h}X|X\in \Gamma ({%
\mathcal{D}}^{\mp })\right\}  \label{formula2}
\end{equation}%
where ${\mathcal{D}}^{+}$ and ${\mathcal{D}}^{\text{-}}$ denote the
eigendistributions of $\tilde{\varphi}$ corresponding to the eigenvalues
corresponding to the eigenvalues $1$ and $-1$, respectively.
\end{corollary}

In the sequel, unless otherwise stated, we will always assume the index of ${%
\mathcal{D}}_{\tilde{h}}(\pm \lambda )$ (in the case $\tilde{\kappa}>-1$)
and of ${\mathcal{D}}_{\tilde{\varphi}\tilde{h}}(\pm \lambda )$ (in the case 
$\tilde{\kappa}<-1$) to be constant.

Being $\tilde{h}$ (in the case $\tilde{\kappa}>-1$) or $\tilde{\varphi}%
\tilde{h}$ (in the case $\tilde{\kappa}<-1$) diagonalizable, one can easily
prove the following lemma. Following similar steps in proof of the theorem (%
\cite{CMEM}, Lemma 3.11), we can give following lemma.

\begin{lemma}
\label{basis} Let $(M,\tilde{\varphi},\xi ,\eta ,\tilde{g})$ be a
paracontact metric $(\tilde{\kappa},\tilde{\mu},\tilde{\nu})$-manifold such
that $\tilde{\kappa}\neq -1$. If $\tilde{\kappa}>-1$ (respectively, $\tilde{%
\kappa}<-1$), then there exists a local orthogonal $\tilde{\varphi}$-basis $%
\{X_{1},\ldots ,X_{n},Y_{1},\ldots ,Y_{n},\xi \}$ of eigenvectors of $\tilde{%
h}$ (respectively, $\tilde{\varphi}\tilde{h}$) such that $X_{1},\ldots
,X_{n}\in \Gamma ({\mathcal{D}}_{\tilde{h}}(\tilde{\lambda}))$
(respectively, $\Gamma ({\mathcal{D}}_{\tilde{\varphi}\tilde{h}}(\tilde{%
\lambda}))$), $Y_{1},\ldots ,Y_{n}\in \Gamma ({\mathcal{D}}_{\tilde{h}}(-%
\tilde{\lambda}))$ (respectively, $\Gamma ({\mathcal{D}}_{\tilde{\varphi}%
\tilde{h}}(-\tilde{\lambda}))$), and 
\begin{equation}
\tilde{g}(X_{i},X_{i})=-\tilde{g}(Y_{i},Y_{i})=\left\{ 
\begin{array}{ll}
1, & \hbox{for $1 \leq i \leq r$} \\ 
-1, & \hbox{for $r+1 \leq i \leq r+s$}%
\end{array}%
\right.  \label{sign}
\end{equation}%
where $r=\emph{index}({\mathcal{D}}_{\tilde{h}}(-\tilde{\lambda}))$
(respectively, $r=\emph{index}({\mathcal{D}}_{\tilde{\varphi}\tilde{h}}(-%
\tilde{\lambda}))$) and $s=n-r=\emph{index}({\mathcal{D}}_{\tilde{h}}(\tilde{%
\lambda}))$ (respectively, $s=\emph{index}({\mathcal{D}}_{\tilde{\varphi}%
\tilde{h}}(\tilde{\lambda}))$).
\end{lemma}

\begin{lemma}
\label{relation k,mu,vu}The following differential equation is satisfied on
every $(2n+1)$-dimensional paracontact metric $(\tilde{\kappa},\tilde{\mu},%
\tilde{\nu})$-manifold $(M,\tilde{\varphi},\xi ,\eta ,\tilde{g})$ :%
\begin{eqnarray}
0 &=&\xi (\tilde{\kappa})(\eta (Y)X-\eta (X)Y)+\xi (\tilde{\mu})(\eta (Y)%
\tilde{h}X-\eta (X)\tilde{h}Y)+\xi (\tilde{\nu})(\eta (Y)\tilde{\varphi}%
\tilde{h}X-\eta (X)\tilde{\varphi}\tilde{h}Y)  \label{DIFEQUATION} \\
&&+X(\tilde{\kappa})\tilde{\varphi}^{2}Y-Y(\tilde{\kappa})\tilde{\varphi}%
^{2}X+X(\tilde{\mu})\tilde{h}Y-Y(\tilde{\mu})\tilde{h}X+X(\tilde{\nu})\tilde{%
\varphi}\tilde{h}Y-Y(\tilde{\nu})\tilde{\varphi}\tilde{h}X.  \notag
\end{eqnarray}
\end{lemma}

\begin{proof}
Differentiating (\ref{PARAKMU}) along an arbitary vector field $Z$ and using
the relation $\tilde{\nabla}\xi =-\tilde{\varphi}+\tilde{\varphi}\tilde{h}$
we find%
\begin{eqnarray*}
\tilde{\nabla}_{Z}\tilde{R}(X,Y)\xi &=&Z(\tilde{\kappa})(\eta (Y)X-\eta
(X)Y)+Z(\tilde{\mu})(\eta (Y)\tilde{h}X-\eta (X)\tilde{h}Y)+Z(\tilde{\nu}%
)(\eta (Y)\tilde{\varphi}\tilde{h}X-\eta (X)\tilde{\varphi}\tilde{h}Y) \\
&&+\tilde{\kappa}[(\eta (\tilde{\nabla}_{Z}Y)-\tilde{g}(Y,\tilde{\varphi}Z)+%
\tilde{g}(Y,\tilde{\varphi}\tilde{h}Z))X+\eta (Y)\tilde{\nabla}_{Z}X \\
&&+(-\eta (\tilde{\nabla}_{Z}X)+\tilde{g}(X,\tilde{\varphi}Z)-\tilde{g}(X,%
\tilde{\varphi}\tilde{h}Z))Y-\eta (X)\tilde{\nabla}_{Z}Y] \\
&&+\tilde{\mu}[(\eta (\tilde{\nabla}_{Z}Y)-\tilde{g}(Y,\tilde{\varphi}Z)+%
\tilde{g}(Y,\tilde{\varphi}\tilde{h}Z))\tilde{h}X+\eta (Y)\tilde{\nabla}_{Z}%
\tilde{h}X \\
&&+(-\eta (\tilde{\nabla}_{Z}X)+\tilde{g}(X,\tilde{\varphi}Z)-\tilde{g}(X,%
\tilde{\varphi}\tilde{h}Z))\tilde{h}Y-\eta (X)\tilde{\nabla}_{Z}\tilde{h}Y]
\\
&&+\tilde{\nu}[(\eta (\tilde{\nabla}_{Z}Y)-\tilde{g}(Y,\tilde{\varphi}Z)+%
\tilde{g}(Y,\tilde{\varphi}\tilde{h}Z))\tilde{\varphi}\tilde{h}X+\eta (Y)%
\tilde{\nabla}_{Z}\tilde{\varphi}\tilde{h}X \\
&&+(-\eta (\tilde{\nabla}_{Z}X)+\tilde{g}(X,\tilde{\varphi}Z)-\tilde{g}(X,%
\tilde{\varphi}\tilde{h}Z))\tilde{\varphi}\tilde{h}Y-\eta (X)\tilde{\nabla}%
_{Z}\tilde{\varphi}\tilde{h}Y].
\end{eqnarray*}%
By using the relation $\tilde{\nabla}\xi =-\tilde{\varphi}+\tilde{\varphi}%
\tilde{h},$ we deduce%
\begin{eqnarray*}
(\tilde{\nabla}_{Z}\tilde{R})(X,Y,\xi ) &=&\tilde{\nabla}_{Z}\tilde{R}%
(X,Y)\xi -\tilde{R}(\tilde{\nabla}_{Z}X,Y)\xi -\tilde{R}(X,\tilde{\nabla}%
_{Z}Y)\xi -\tilde{R}(X,Y)\tilde{\nabla}_{Z}\xi \\
&=&Z(\tilde{\kappa})(\eta (Y)X-\eta (X)Y)+Z(\tilde{\mu})(\eta (Y)\tilde{h}%
X-\eta (X)\tilde{h}Y)+Z(\tilde{\nu})(\eta (Y)\tilde{\varphi}\tilde{h}X-\eta
(X)\tilde{\varphi}\tilde{h}Y) \\
&&+\tilde{\kappa}[\tilde{g}(Y,-\tilde{\varphi}Z+\tilde{\varphi}\tilde{h}Z)X+%
\tilde{g}(X,\tilde{\varphi}Z-\tilde{\varphi}\tilde{h}Z)Y] \\
&&+\tilde{\mu}[\tilde{g}(Y,-\tilde{\varphi}Z+\tilde{\varphi}\tilde{h}Z)%
\tilde{h}X+\tilde{g}(X,\tilde{\varphi}Z-\tilde{\varphi}\tilde{h}Z)\tilde{h}Y
\\
&&+\eta (Y)(\tilde{\nabla}_{Z}\tilde{h})X-\eta (X)(\tilde{\nabla}_{Z}\tilde{h%
})Y] \\
&&+\tilde{\nu}[\tilde{g}(Y,-\tilde{\varphi}Z+\tilde{\varphi}\tilde{h}Z)%
\tilde{\varphi}\tilde{h}X+\tilde{g}(X,\tilde{\varphi}Z-\tilde{\varphi}\tilde{%
h}Z))\tilde{\varphi}\tilde{h}Y \\
&&+\eta (Y)(\tilde{\nabla}_{Z}\tilde{\varphi}\tilde{h})X-\eta (X)(\tilde{%
\nabla}_{Z}\tilde{\varphi}\tilde{h})Y]+\tilde{R}(X,Y)\tilde{\varphi}Z-\tilde{%
R}(X,Y)\tilde{\varphi}\tilde{h}Z.
\end{eqnarray*}%
Using the second Bianchi identity in the last relation, we obtain%
\begin{eqnarray*}
0 &=&Z(\tilde{\kappa})(\eta (Y)X-\eta (X)Y)+Z(\tilde{\mu})(\eta (Y)\tilde{h}%
X-\eta (X)\tilde{h}Y)+Z(\tilde{\nu})(\eta (Y)\tilde{\varphi}\tilde{h}X-\eta
(X)\tilde{\varphi}\tilde{h}Y) \\
&&+X(\tilde{\kappa})(\eta (Z)Y-\eta (Y)Z)+X(\tilde{\mu})(\eta (Z)\tilde{h}%
Y-\eta (Y)\tilde{h}Z)+X(\tilde{\nu})(\eta (Z)\tilde{\varphi}\tilde{h}Y-\eta
(Y)\tilde{\varphi}\tilde{h}Z) \\
&&+Y(\tilde{\kappa})(\eta (X)Z-\eta (Z)X)+Y(\tilde{\mu})(\eta (X)\tilde{h}%
Z-\eta (Z)\tilde{h}X)+Y(\tilde{\nu})(\eta (X)\tilde{\varphi}\tilde{h}Z-\eta
(Z)\tilde{\varphi}\tilde{h}X) \\
&&+2\tilde{\kappa}[\tilde{g}(\tilde{\varphi}Y,Z)X+\tilde{g}(\tilde{\varphi}%
Z,X)Y+\tilde{g}(\tilde{\varphi}X,Y)Z] \\
&&+\tilde{\mu}[2\tilde{g}(\tilde{\varphi}Y,Z)\tilde{h}X+\eta (Z)((\tilde{%
\nabla}_{X}\tilde{h})Y-(\tilde{\nabla}_{Y}\tilde{h})X)+2\tilde{g}(\tilde{%
\varphi}Z,X)\tilde{h}Y+\eta (X)((\tilde{\nabla}_{Y}\tilde{h})Z-(\tilde{\nabla%
}_{Z}\tilde{h})Y) \\
&&+2\tilde{g}(\tilde{\varphi}X,Y)\tilde{h}Z+\eta (Y)((\tilde{\nabla}_{Z}%
\tilde{h})X-(\tilde{\nabla}_{X}\tilde{h})Z)] \\
&&+\tilde{\nu}[2\tilde{g}(\tilde{\varphi}Y,Z)\tilde{\varphi}\tilde{h}X+\eta
(Z)((\tilde{\nabla}_{X}\tilde{\varphi}\tilde{h})Y-(\tilde{\nabla}_{Y}\tilde{%
\varphi}\tilde{h})X)+2\tilde{g}(\tilde{\varphi}Z,X)\tilde{\varphi}\tilde{h}%
Y+\eta (X)((\tilde{\nabla}_{Y}\tilde{\varphi}\tilde{h})Z-(\tilde{\nabla}_{Z}%
\tilde{\varphi}\tilde{h})Y) \\
&&+2\tilde{g}(\tilde{\varphi}X,Y)\tilde{\varphi}\tilde{h}Z+\eta (Y)((\tilde{%
\nabla}_{Z}\tilde{\varphi}\tilde{h})X-(\tilde{\nabla}_{X}\tilde{\varphi}%
\tilde{h})Z)] \\
&&+\tilde{R}(X,Y)\tilde{\varphi}Z+\tilde{R}(Y,Z)\tilde{\varphi}X+\tilde{R}%
(Z,X)\tilde{\varphi}Y \\
&&-\tilde{R}(X,Y)\tilde{\varphi}\tilde{h}Z-\tilde{R}(Y,Z)\tilde{\varphi}%
\tilde{h}X-\tilde{R}(Z,X)\tilde{\varphi}\tilde{h}Y.
\end{eqnarray*}%
for all $X,Y,Z\in \Gamma (TM)$. Putting $\xi $ instead of $Z$ in the last
relation, we obtain%
\begin{eqnarray*}
0 &=&\xi (\tilde{\kappa})(\eta (Y)X-\eta (X)Y)+\xi (\tilde{\mu})(\eta (Y)%
\tilde{h}X-\eta (X)\tilde{h}Y)+\xi (\tilde{\nu})(\eta (Y)\tilde{\varphi}%
\tilde{h}X-\eta (X)\tilde{\varphi}\tilde{h}Y) \\
&&+X(\tilde{\kappa})Y-\eta (Y)\xi +X(\tilde{\mu})\tilde{h}Y+X(\tilde{\nu})%
\tilde{\varphi}\tilde{h}Y \\
&&+Y(\tilde{\kappa})(\eta (X)\xi -X)-Y(\tilde{\mu})\tilde{h}X-Y(\tilde{\nu})%
\tilde{\varphi}\tilde{h}X+2\tilde{\kappa}\tilde{g}(\tilde{\varphi}X,Y)\xi \\
&&+\tilde{\mu}[(\tilde{\nabla}_{X}\tilde{h})Y-(\tilde{\nabla}_{Y}\tilde{h}%
)X+\eta (X)((\tilde{\nabla}_{Y}\tilde{h})\xi -(\tilde{\nabla}_{\xi }\tilde{h}%
)Y)+\eta (Y)((\tilde{\nabla}_{\xi }\tilde{h})X-(\tilde{\nabla}_{X}\tilde{h}%
)\xi )] \\
&&+\tilde{\nu}[(\tilde{\nabla}_{X}\tilde{\varphi}\tilde{h})Y-(\tilde{\nabla}%
_{Y}\tilde{\varphi}\tilde{h})X+\eta (X)((\tilde{\nabla}_{Y}\tilde{\varphi}%
\tilde{h})\xi -(\tilde{\nabla}_{\xi }\tilde{\varphi}\tilde{h})Y) \\
&&+\eta (Y)((\tilde{\nabla}_{\xi }\tilde{\varphi}\tilde{h})X-(\tilde{\nabla}%
_{X}\tilde{\varphi}\tilde{h})\xi )] \\
&&+\tilde{R}(Y,\xi )\tilde{\varphi}X+\tilde{R}(\xi ,X)\tilde{\varphi}Y-%
\tilde{R}(Y,\xi )\tilde{\varphi}\tilde{h}X-\tilde{R}(\xi ,X)\tilde{\varphi}%
\tilde{h}Y.
\end{eqnarray*}%
Using (\ref{NAMLA X H}),(\ref{NAMLAFIH}) and (\ref{RZETAXY}) in the last
relation we finally get (\ref{DIFEQUATION}) and it completes the proof.
\end{proof}

By Corollary 3.6 and Lemma 3.7, one can obtain the proof of following
theorem by using a similar method of (\cite{KMP}, Theorem 4.1).

\begin{theorem}
\label{tson}Let $(M,\tilde{\varphi},\xi ,\eta ,\tilde{g})$ be a $2n+1$%
-dimensional paracontact metric $(\tilde{\kappa},\tilde{\mu},\tilde{\nu})$%
-manifold with $\tilde{\kappa}\neq -1$ and $n>1$. Then $M$ is a paracontact
metric $(\tilde{\kappa},\tilde{\mu})$-manifold, i.e $\ \tilde{\kappa},\tilde{%
\mu}$ are constants and $\tilde{\nu}$ is the zero function.
\end{theorem}

\begin{remark}
\label{rem1}By (\ref{H2}), for $\tilde{\kappa}=-1$ we obtain $\tilde{h}%
^{2}=0.$ Since the metric $\tilde{g}$ is not positive definite we can not
conclude that $\tilde{h}=0$ and the manifold is para-Sasakian. The following
question comes up naturally. Do there exist a paracontact metric $(-1,\tilde{%
\mu})$ manifold with dimension greater than $3$ satisfying $\tilde{h}^{2}=0$
but $\tilde{h}\neq 0$? B. Cappelletti Montano and L. Di Terlizzi \cite{MOTE}
\ and V. Martin-Molina \cite{VMM} gave an affirmative answer for $n>1.$
\end{remark}

In $(\kappa ,\mu ,\nu $ )-contact metric case, T.Koufogiorgos et al. \cite%
{KMP} proved that if dimension greater than $3$ then it is either a Sasakian
manifold or a $(\kappa ,\mu )$-contact metric manifold. Now we will give an
example of paracontact metric $(-1,\tilde{\mu},\tilde{\nu}\neq 0)$-manifold
for five dimensional which have no contact metric counterpart $(\tilde{h}%
^{2}=0$ but $\tilde{h}\neq 0)$. This is first example of this type manifold .

\begin{example}
\label{ex5dim}Let $\mathfrak{g}$ be the $5$-dimensional Lie algebra with
basis $X_{1},X_{2},Y_{1},Y_{2},\xi $ and Lie brackets defined by$\ \ \ $%
\begin{equation*}
\begin{array}{cc}
\left[ \xi ,X_{1}\right] =Y_{1}~, & ~\left[ \xi ,Y_{1}\right] =0 \\ 
\left[ \xi ,X_{2}\right] =Y_{2}~~,~~~ & \left[ \xi ,X_{2}\right] =Y_{2}~~~~~
\\ 
\left[ X_{1},Y_{1}\right] =2\xi -Y_{1}~~, & ~~~\left[ X_{2},Y_{2}\right] =%
\frac{2}{3}Y_{1}-Y_{2}-2\xi \\ 
\left[ X_{1},X_{2}\right] =-\frac{1}{3}X_{1}-\frac{1}{3}X_{2}, & ~\left[
Y_{1},Y_{2}\right] =0 \\ 
\left[ X_{1},Y_{2}\right] =-\frac{2}{3}Y_{1}-\frac{1}{3}Y_{2}, & ~\left[
Y_{1},X_{2}\right] =\frac{1}{3}Y_{1}.%
\end{array}%
\end{equation*}
\end{example}

Let $\mathcal{G}$ be a Lie group whose Lie algebra is $\mathfrak{g}$. On $%
\mathcal{G}$ we define a left-invariant paracontact metric structure $(%
\tilde{\varphi},\xi ,\eta ,\tilde{g})$ by setting $\tilde{\varphi}\xi =0$
and $\tilde{\varphi}X_{i}=X_{i},~\tilde{\varphi}Y_{i}=-Y_{i},\eta (\xi
)=1,\eta (X_{i})=\eta (Y_{i})=0$ and whose only non-vanishing components of
the metric are $~\tilde{g}(\xi ,\xi )=\tilde{g}(X_{1},Y_{1})=1,~\tilde{g}%
(X_{2},Y_{2})=-1$ for all$~i=1,2.$ Therefore, $\tilde{h}X_{i}=Y_{i}$ and $~%
\tilde{h}Y_{i}=0,~i=1,2$, so $\tilde{h}^{2}=0$ but $\tilde{h}\neq 0$ and
rank($\tilde{h}$)$=2.$Moreover, one can verify that ($\mathcal{G},\tilde{%
\varphi},\xi ,\eta ,\tilde{g})$ is a paracontact metric $(\tilde{\kappa},%
\tilde{\mu},\tilde{\nu})$-manifold with $\tilde{\kappa}=-1,\tilde{\mu}=-1,%
\tilde{\nu}=-3.$

\section{Classification of the $3$-dimensional Paracontact metric $(\tilde{%
\protect\kappa},\tilde{\protect\mu},\tilde{\protect\nu})$-manifolds}

\label{fourth}

In this section, we analyze the different possibilities for the tensor field 
$\tilde{h}$. Hence one can recognize the differences between the contact and
paracontact cases by looking at the possible Jordan forms of the tensor
field $\tilde{h}.$

A self-adjoint linear operator $A$ of a Riemannian manifold is always
diagonalizable, but this is not the case for a self-adjoint linear operator $%
A$ of a Lorentzian manifold. It is known (\cite{a.z.petrov}, pp. 50-55) that
self-adjoint linear operator of a vector space with a Lorentzian inner
product can be put into four possible canonical forms. In particular, the
matrix representation $\tilde{g}$ of the induced metric on $M_{1}^{3}$ is of
Lorentz type, so the self-adjoint linear $A$ of $M_{1}^{3}$ can be put into
one of the following four forms with respect to frames $\left\{
e_{1},e_{2},e_{3}\right\} $ at $T_{p}M_{1}^{3}$ where $T_{p}M_{1}^{3}$ is
called tangent space to $M$ at $p$ \cite{MAGID},\cite{ONEILL}.

(I) $\ \ \ \ \ \ \ \ \ \ \ \ \ \ \ \ A=\left( 
\begin{array}{ccc}
\lambda _{1} & 0 & 0 \\ 
0 & \lambda _{2} & 0 \\ 
0 & 0 & \lambda _{3}%
\end{array}%
\right) ,$ \ \ \ \ $\ \ \ \ \ \ \tilde{g}=\left( 
\begin{array}{ccc}
-1 & 0 & 0 \\ 
0 & 1 & 0 \\ 
0 & 0 & 1%
\end{array}%
\right) ,$

\bigskip

(II) \ \ \ \ \ \ \ \ \ \ \ \ \ \ \ \ $A=\left( 
\begin{array}{ccc}
\lambda & 0 & 0 \\ 
1 & \lambda & 0 \\ 
0 & 0 & \lambda _{3}%
\end{array}%
\right) ,$ \ \ \ \ \ \ \ \ \ \ \ \ \ \ $\tilde{g}=\left( 
\begin{array}{ccc}
0 & 1 & 0 \\ 
1 & 0 & 0 \\ 
0 & 0 & 1%
\end{array}%
\right) $,\ \ 

\bigskip

(III) \ \ \ \ \ \ \ \ \ \ \ \ \ \ $\ A=\left( 
\begin{array}{ccc}
\gamma & -\lambda & 0 \\ 
\lambda & \gamma & 0 \\ 
0 & 0 & \lambda _{3}%
\end{array}%
\right) ,$ \ \ \ \ $\ \ \ \ \ \ $\ $\ \ \tilde{g}=\left( 
\begin{array}{ccc}
-1 & 0 & 0 \\ 
0 & 1 & 0 \\ 
0 & 0 & 1%
\end{array}%
\right) ,\lambda \neq 0$,\ \ \ \ \ \ \ \ \ \ 

\bigskip

(IV) $\ \ \ \ \ \ \ \ \ \ \ \ \ \ \ \ A=\left( 
\begin{array}{ccc}
\lambda & 0 & 0 \\ 
0 & \lambda & 1 \\ 
1 & 0 & \lambda%
\end{array}%
\right) ,$ \ \ \ \ \ \ \ \ \ \ \ \ \ \ \ $\tilde{g}=\left( 
\begin{array}{ccc}
0 & 1 & 0 \\ 
1 & 0 & 0 \\ 
0 & 0 & 1%
\end{array}%
\right) $.

The matrices $\tilde{g}$ for cases (I) and (III) are with respect to an
orthonormal basis of \ $T_{p}M_{1}^{3},$ whereas for cases (II) and (IV) are
with respect to a pseudo-orthonormal basis. This is a basis $\left\{
e_{1},e_{2},e_{3}\right\} $ of \ $T_{p}M_{1}^{3}$ satisfying $\tilde{g}%
(e_{1},e_{1})=\tilde{g}(e_{2},e_{2})=\tilde{g}(e_{1},e_{3})=\tilde{g}%
(e_{2},e_{3})=0$ and $\tilde{g}(e_{1},e_{2})=\tilde{g}(e_{3},e_{3})=1.$\ 

Next, we recall that the curvature tensor of a $3$-dimensional
pseudo-Riemannian manifold satisfies 
\begin{equation}
\tilde{R}(X,Y)Z=\tilde{g}(Y,Z)\tilde{Q}X-\tilde{g}(X,Z)\tilde{Q}Y+\tilde{g}(%
\tilde{Q}Y,Z)X-\tilde{g}(\tilde{Q}X,Z)Y-\frac{r}{2}(\tilde{g}(Y,Z)X-\tilde{g}%
(X,Z)Y).  \label{THREE DIM CURVATURE}
\end{equation}%
\ \ 

\textbf{The tensor }$\tilde{h}$\textbf{\ has the canonical form (I).} \ Let $%
(M,\tilde{\varphi},\xi ,\eta ,\tilde{g})$ be a $3$-dimensional paracontact
metric manifold . 
\begin{eqnarray*}
U_{1} &=&\left\{ p\in M\mid \tilde{h}(p)\neq 0\right\} \subset M \\
U_{2} &=&\left\{ p\in M\mid \tilde{h}(p)=0,\text{in a neighborhood of p}%
\right\} \subset M
\end{eqnarray*}%
That $\tilde{h}$ is a smooth function on $M$ implies $U_{1}\cup U_{2}$ is an
open and dense subset of $M$, so any property satisfied in $U_{1}\cup U_{2}$
is also satisfied in $M.$ For any point $p\in U_{1}\cup U_{2}$ there exists
a local orthonormal $\tilde{\varphi}$-basis $\{\tilde{e},\tilde{\varphi}%
\tilde{e},\xi \}$ of smooth eigenvectors of $\tilde{h}$ in a neighborhood of 
$p$, where $-\tilde{g}(\tilde{e},\tilde{e})=\tilde{g}(\tilde{\varphi}\tilde{e%
},\tilde{\varphi}\tilde{e})=\tilde{g}(\xi ,\xi )=1.$ On $U_{1}$ we put $%
\tilde{h}\tilde{e}=\tilde{\lambda}\tilde{e},$ where $\tilde{\lambda}$ is a
non-vanishing smooth function. Since $tr\tilde{h}=0$, we have $\tilde{h}%
\tilde{\varphi}\tilde{e}=-\tilde{\lambda}\tilde{\varphi}\tilde{e}$. The
eigenvalue function $\tilde{\lambda}$ is continuos on$\ M$ and smooth on $%
U_{1}\cup U_{2}.$ So, $\tilde{h}$ \ has following form%
\begin{equation}
\left( 
\begin{array}{ccc}
\tilde{\lambda} & 0 & 0 \\ 
0 & -\tilde{\lambda} & 0 \\ 
0 & 0 & 0%
\end{array}%
\right)  \label{A1}
\end{equation}%
respect to local orthonormal $\tilde{\varphi}$-basis $\{\tilde{e},\tilde{%
\varphi}\tilde{e},\xi \}$. In this case,\ we will say the operator $\tilde{h}
$ is of $\mathfrak{h}_{1}$ \textit{type}$.$ Using same method with \cite{KMP}
and \cite{PEROONE03}, we have

\begin{lemma}
\label{K>-1}Let $(M,\tilde{\varphi},\xi ,\eta ,\tilde{g})$ be a $3$%
-dimensional paracontact metric manifold with $\tilde{h}$ of $\mathfrak{h}%
_{1}$\emph{\ }\textit{type}$.$Then for the covariant derivative on $U_{1}$
the following equations are valid 
\begin{eqnarray}
i)\text{ }\tilde{\nabla}_{\tilde{e}}\tilde{e} &=&\frac{1}{2\tilde{\lambda}}%
\left[ \tilde{\sigma}(\tilde{e})-(\tilde{\varphi}\tilde{e})(\tilde{\lambda})%
\right] \tilde{\varphi}\tilde{e},\text{ \ }ii)\text{ }\tilde{\nabla}_{\tilde{%
e}}\tilde{\varphi}\tilde{e}=\frac{1}{2\tilde{\lambda}}\left[ \tilde{\sigma}(%
\tilde{e})-(\tilde{\varphi}\tilde{e})(\tilde{\lambda})\right] \tilde{e}+(1-%
\tilde{\lambda})\xi ,\text{ }  \notag \\
\text{\ }iii)\text{ }\tilde{\nabla}_{\tilde{e}}\xi &=&(\tilde{\lambda}-1)%
\tilde{\varphi}\tilde{e},  \notag \\
iv)\text{ }\tilde{\nabla}_{\tilde{\varphi}\tilde{e}}\tilde{e} &=&-\frac{1}{2%
\tilde{\lambda}}\left[ \tilde{\sigma}(\tilde{\varphi}\tilde{e})+\tilde{e}(%
\tilde{\lambda})\right] \tilde{\varphi}\tilde{e}-(\tilde{\lambda}+1)\xi ,%
\text{ \ }v)\text{ }\tilde{\nabla}_{\tilde{\varphi}\tilde{e}}\tilde{\varphi}%
\tilde{e}=-\frac{1}{2\tilde{\lambda}}\left[ \tilde{\sigma}(\tilde{\varphi}%
\tilde{e})+\tilde{e}(\tilde{\lambda})\right] \tilde{e},\text{ \ }  \notag \\
vi)\text{ }\tilde{\nabla}_{\tilde{\varphi}\tilde{e}}\xi &=&-(\tilde{\lambda}%
+1)\tilde{e},  \notag \\
vii)\text{ }\tilde{\nabla}_{\xi }\tilde{e} &=&b\tilde{\varphi}\tilde{e},%
\text{ \ }viii)\text{ }\tilde{\nabla}_{\xi }\tilde{\varphi}\tilde{e}=b\tilde{%
e},\text{ \ \ }  \notag \\
ix)\text{ }[\tilde{e},\xi ] &=&(\tilde{\lambda}-1-b)\tilde{\varphi}\tilde{e},%
\text{ \ }x)\text{ }[\tilde{\varphi}\tilde{e},\xi ]=(-\tilde{\lambda}-1-b)%
\tilde{e},\text{ \ }  \label{irem0} \\
xi)\text{ }[\tilde{e},\tilde{\varphi}\tilde{e}] &=&\frac{1}{2\tilde{\lambda}}%
\left[ \tilde{\sigma}(\tilde{e})-(\tilde{\varphi}\tilde{e})(\tilde{\lambda})%
\right] \tilde{e}+\frac{1}{2\tilde{\lambda}}\left[ \tilde{\sigma}(\tilde{%
\varphi}\tilde{e})+\tilde{e}(\lambda )\right] \tilde{\varphi}\tilde{e}+2\xi ,
\notag
\end{eqnarray}%
where 
\begin{equation*}
\text{ }b=\tilde{g}(\text{ }\tilde{\nabla}_{\xi }\tilde{e},\tilde{\varphi}%
\tilde{e}),\text{ }\tilde{\sigma}=\tilde{S}(\xi ,.)_{\ker \eta }\text{\ .}
\end{equation*}
\end{lemma}

\begin{proposition}
\label{d2}Let $(M,\tilde{\varphi},\xi ,\eta ,\tilde{g})$ be a $3$%
-dimensional paracontact metric manifold. If $\tilde{h}$ is $\mathfrak{h}%
_{1} $ type then on $U_{1}$ we have%
\begin{equation}
\tilde{\nabla}_{\xi }\tilde{h}=-2b\tilde{h}\tilde{\varphi}+\xi (\tilde{%
\lambda})s,  \label{irem4}
\end{equation}%
where $s$ is the $(1,1)$-type tensor defined by $s\xi =0,$ $s\tilde{e}=%
\tilde{e},$ \ $s\tilde{\varphi}\tilde{e}=-\tilde{\varphi}\tilde{e}.$
\end{proposition}

\begin{proof}
Using (\ref{irem0}), we get%
\begin{eqnarray*}
(\tilde{\nabla}_{\xi }\tilde{h})\xi &=&0=(-2b\tilde{h}\tilde{\varphi}+\xi (%
\tilde{\lambda})s)\xi , \\
(\tilde{\nabla}_{\xi }\tilde{h})\tilde{e} &=&2\tilde{\lambda}b\tilde{\varphi}%
\tilde{e}+\xi (\tilde{\lambda})\tilde{e}=(-2b\tilde{h}\tilde{\varphi}+\xi (%
\tilde{\lambda})s)\tilde{e}, \\
(\tilde{\nabla}_{\xi }\tilde{h})\tilde{\varphi}\tilde{e} &=&-2\tilde{\lambda}%
b\tilde{e}-\xi (\tilde{\lambda})\tilde{\varphi}\tilde{e}=(-2b\tilde{h}\tilde{%
\varphi}+\xi (\tilde{\lambda})s)\tilde{\varphi}\tilde{e}.
\end{eqnarray*}

So, the last equations complete the proof of (\ref{irem4}).
\end{proof}

\begin{proposition}
\label{h-fi}Let $(M,\tilde{\varphi},\xi ,\eta ,\tilde{g})$ be a $3$%
-dimensional paracontact metric manifold with $\tilde{h}$ of $\mathfrak{h}%
_{1}$ type$.$Then the following equation holds on $M.$%
\begin{equation}
\tilde{h}^{2}-\tilde{\varphi}^{2}=\frac{\tilde{S}(\xi ,\xi )}{2}\tilde{%
\varphi}^{2}.  \label{irem5}
\end{equation}
\end{proposition}

\begin{proof}
Using (\ref{RICCI ZETA}), we have $\tilde{S}(\xi ,\xi )=2(\tilde{\lambda}%
^{2}-1)$. After calculating $\tilde{h}^{2}-\tilde{\varphi}^{2}$ with respect
to the basis components, we get%
\begin{equation}
\tilde{h}^{2}\xi -\tilde{\varphi}^{2}\xi =\frac{\tilde{S}(\xi ,\xi )}{2}%
\tilde{\varphi}^{2}\xi =0,\text{ \ }\tilde{h}^{2}\tilde{e}-\tilde{\varphi}%
^{2}\tilde{e}=\frac{\tilde{S}(\xi ,\xi )}{2}\tilde{\varphi}^{2}\tilde{e},%
\text{ \ }\tilde{h}^{2}\tilde{\varphi}\tilde{e}-\tilde{\varphi}^{3}\tilde{e}=%
\frac{\tilde{S}(\xi ,\xi )}{2}\tilde{\varphi}^{2}\tilde{\varphi}\tilde{e}.
\label{irem6}
\end{equation}%
Thus the last equation completes the proof of (\ref{irem5}).
\end{proof}

\begin{lemma}
\label{Q}Let $(M,\tilde{\varphi},\xi ,\eta ,\tilde{g})$ be a $3$-dimensional
paracontact metric manifold. If $\tilde{h}$ is $\mathfrak{h}_{1}$ type then
the Ricci operator $\tilde{Q}$ is given by%
\begin{equation}
\tilde{Q}=a_{1}I+b_{1}\eta \otimes \xi -\tilde{\varphi}(\tilde{\nabla}_{\xi }%
\tilde{h})+\tilde{\sigma}(\tilde{\varphi}^{2})\otimes \xi -\tilde{\sigma}(%
\tilde{e})\eta \otimes \tilde{e}+\tilde{\sigma}(\tilde{\varphi}\tilde{e}%
)\eta \otimes \tilde{\varphi}\tilde{e},  \label{irem7}
\end{equation}%
where $a_{1}$ and $b_{1}$ are smooth functions defined by $a_{1}$ $=1-\tilde{%
\lambda}^{2}+\frac{r}{2}$ and $b_{1}=3(\tilde{\lambda}^{2}-1)-\frac{r}{2},$
respectively. Moreover the components of the Ricci operator $\tilde{Q}$ are
given by
\end{lemma}

\begin{eqnarray}
\tilde{Q}\xi &=&(a_{1}+b_{1})\xi -\tilde{\sigma}(\tilde{e})\tilde{e}+\tilde{%
\sigma}(\tilde{\varphi}\tilde{e})\tilde{\varphi}\tilde{e},  \notag \\
\tilde{Q}\tilde{e} &=&\tilde{\sigma}(\tilde{e})\xi +(a_{1}-2b\tilde{\lambda})%
\tilde{e}-\xi (\tilde{\lambda})\tilde{\varphi}\tilde{e},  \label{irem7a} \\
\tilde{Q}\tilde{\varphi}\tilde{e} &=&\tilde{\sigma}(\tilde{\varphi}\tilde{e}%
)\xi +\xi (\tilde{\lambda})\tilde{e}+(a_{1}+2b\tilde{\lambda})\tilde{\varphi}%
\tilde{e}.  \notag
\end{eqnarray}

\begin{proof}
From (\ref{THREE DIM CURVATURE}), we have%
\begin{equation*}
\tilde{l}X=\tilde{R}(X,\xi )\xi =\tilde{S}(\xi ,\xi )X-\tilde{S}(X,\xi )\xi +%
\tilde{Q}X-\eta (X)\tilde{Q}\xi -\frac{r}{2}(X-\eta (X)\xi ),
\end{equation*}%
where $\tilde{l}$ denotes the Jacobi operator and $X$ is a vector field.
Using (\ref{irem000}), the last equation implies%
\begin{equation*}
\tilde{Q}X=-\tilde{\varphi}^{2}X+\tilde{h}^{2}X-\tilde{\varphi}(\tilde{\nabla%
}_{\xi }\tilde{h})X-\tilde{S}(\xi ,\xi )X+\tilde{S}(X,\xi )\xi +\eta (X)%
\tilde{Q}\xi +\frac{r}{2}(X-\eta (X)\xi ).
\end{equation*}%
Since $\tilde{S}(X,\xi )=\tilde{S}(\tilde{\varphi}^{2}X,\xi )+\eta (X)\tilde{%
S}(\xi ,\xi ),$ we have%
\begin{equation}
\tilde{Q}X=\frac{\tilde{S}(\xi ,\xi )}{2}\tilde{\varphi}^{2}X-\tilde{\varphi}%
(\tilde{\nabla}_{\xi }\tilde{h})X-\tilde{S}(\xi ,\xi )X+\tilde{S}(\tilde{%
\varphi}^{2}X,\xi )\xi +\eta (X)\tilde{S}(\xi ,\xi )\xi +\eta (X)\tilde{Q}%
\xi +\frac{r}{2}\tilde{\varphi}^{2}X.  \label{irem8}
\end{equation}%
One can easily prove that 
\begin{equation}
\tilde{Q}\xi =-\tilde{\sigma}(\tilde{e})\tilde{e}+\tilde{\sigma}(\tilde{%
\varphi}\tilde{e})\tilde{\varphi}\tilde{e}+\tilde{S}(\xi ,\xi )\xi .
\label{irem9}
\end{equation}%
Using (\ref{irem9}) in (\ref{irem8}), we have%
\begin{eqnarray}
\tilde{Q}X &=&\left( 1-\tilde{\lambda}^{2}+\frac{r}{2}\right) X+\left( 3(%
\tilde{\lambda}^{2}-1)-\frac{r}{2}\right) \eta (X)\xi  \label{irem10} \\
&&-\tilde{\varphi}(\tilde{\nabla}_{\xi }\tilde{h})X+\tilde{\sigma}(\tilde{%
\varphi}^{2}X)\xi -\eta (X)\tilde{\sigma}(\tilde{e})\tilde{e}+\eta (X)\tilde{%
\sigma}(\tilde{\varphi}\tilde{e})\tilde{\varphi}\tilde{e},  \notag
\end{eqnarray}%
for arbitrary vector field $X$. Hence, the proof follows from (\ref{irem10}%
). By (\ref{irem4}) and (\ref{irem10}) we get (\ref{irem7a}).
\end{proof}

\textbf{The tensor }$\tilde{h}$\textbf{\ has the canonical form (II).} Let $%
(M,\tilde{\varphi},\xi ,\eta ,\tilde{g})$ be a $3$-dimensional paracontact
metric manifold and $p$ is a point of $M$. Then there exists a local
pseudo-orthonormal basis $\{e_{1},e_{2},\xi \}$ in a neighborhood of $p$
where $\tilde{g}(e_{1},e_{1})=\tilde{g}(e_{2},e_{2})=\tilde{g}(e_{1},\xi )=%
\tilde{g}(e_{2},\xi )=0$ and $\tilde{g}(e_{1},e_{2})=1.$

\begin{lemma}
\label{II}Let $\mathcal{U}$ be the open subset of $M$ where $\tilde{h}\neq 0$%
. For every $p\in \mathcal{U}$ there exists an open neighborhood of $p$ such
that $\tilde{h}e_{1}=e_{2},\tilde{h}e_{2}=0,\tilde{h}\xi =0$ and $\tilde{%
\varphi}e_{1}=\pm e_{1},$ $\tilde{\varphi}e_{2}=\mp e_{2}.$
\end{lemma}

\begin{proof}
Since the tensor $\tilde{h}$ has canonical form (II) (with respect to a
pseudo-orthonormal basis $\{e_{1},e_{2},\xi \})$ then $\tilde{h}e_{1}=\tilde{%
\lambda}e_{1}+e_{2},$ \ $\tilde{h}e_{2}=\tilde{\lambda}e_{2},$ $\tilde{h}\xi
=0.$ Since $\tilde{h}\xi =0$ and $tr(h)=0$ we have $\tilde{\lambda}=0.$ On
the other hand using the anti-symmetry tensor field $\tilde{\varphi}$ of
type $(1,1),$ with respect to the pseudo-orthonormal basis $%
\{e_{1},e_{2},\xi \}$, takes the form%
\begin{equation}
\left( 
\begin{array}{ccc}
\tilde{\varphi}_{11} & \tilde{\varphi}_{12} & 0 \\ 
\tilde{\varphi}_{21} & \tilde{\varphi}_{22} & 0 \\ 
0 & 0 & 0%
\end{array}%
\right) .  \label{MATRIX FI}
\end{equation}%
Using (\ref{G METRIC}) and (\ref{MATRIX FI}) we have%
\begin{eqnarray}
\tilde{g}(\tilde{\varphi}e_{1},e_{1}) &=&0=\tilde{\varphi}_{21}\text{ and \ }%
\tilde{g}(\tilde{\varphi}e_{2},e_{2})=0=\tilde{\varphi}_{12},
\label{MATRIXF2} \\
\tilde{g}(\tilde{\varphi}e_{1},e_{2}) &=&\tilde{\varphi}_{11}=-\tilde{g}%
(e_{1},\tilde{\varphi}e_{2})=-\tilde{\varphi}_{22}\text{ }  \notag
\end{eqnarray}%
Other hand we get%
\begin{equation}
\tilde{g}(\tilde{\varphi}e_{1},\tilde{\varphi}e_{2})\text{ }=\tilde{\varphi}%
_{11}\tilde{\varphi}_{22}=-g(e_{1},e_{2})=1.  \label{MATRIXF3}
\end{equation}%
From \ two last equation , $\tilde{\varphi}_{22}=\mp 1.$ This completes the
proof.
\end{proof}

Hence the tensor $\tilde{h}$ has the form $\left( 
\begin{array}{ccc}
0 & 0 & 0 \\ 
1 & 0 & 0 \\ 
0 & 0 & 0%
\end{array}%
\right) $ relative a pseudo-orthonormal basis $\{e_{1},e_{2},e_{3}\}$. In
this case, we call $\tilde{h}$ is of $\mathfrak{h}_{2}$ \textit{type}.

\begin{remark}
\label{with}Without loss of generality, we can assume that $\tilde{\varphi}%
e_{1}=e_{1}$ $\tilde{\varphi}e_{2}=-e_{2}.$ Moreover one can easily get $%
\tilde{h}^{2}=0$ but $\tilde{h}\neq 0.$
\end{remark}

\begin{lemma}
\label{de2}Let $(M,\tilde{\varphi},\xi ,\eta ,\tilde{g})$ be a $3$%
-dimensional paracontact metric manifold with $\tilde{h}$ of $\mathfrak{h}%
_{2}$ type. Then for the covariant derivative on $\mathcal{U}$ the following
equations are valid 
\begin{eqnarray}
i)\text{ }\tilde{\nabla}_{e_{1}}e_{1} &=&-b_{2}e_{1}+\xi ,\text{ \ }ii)\text{
}\tilde{\nabla}_{e_{1}}e_{2}=b_{2}e_{2}+\xi ,\text{ \ }iii)\text{ }\tilde{%
\nabla}_{e_{1}}\xi =-e_{1}-e_{2},  \notag \\
iv)\text{ }\tilde{\nabla}_{e_{2}}e_{1} &=&-\tilde{b}_{2}e_{1}-\xi ,\text{ \ }%
v)\text{ }\tilde{\nabla}_{e_{2}}e_{2}=\tilde{b}_{2}e_{2},\text{ \ }vi)\text{ 
}\tilde{\nabla}_{e_{2}}\xi =e_{2},  \notag \\
vii)\text{ }\tilde{\nabla}_{\xi }e_{1} &=&a_{2}e_{1},\text{ \ }viii)\text{ }%
\tilde{\nabla}_{\xi }e_{2}=-a_{2}e_{2},\text{ \ \ }  \notag \\
ix)\text{ }[e_{1},\xi ] &=&-(1+a_{2})e_{1}-e_{2},\text{ \ }x)\text{ }%
[e_{2},\xi ]=(1+a_{2})e_{2},\text{ \ }  \label{iremm0} \\
xi)\text{ }[e_{1},e_{2}] &=&\tilde{b}_{2}e_{1}+b_{2}e_{2}+2\xi .  \notag
\end{eqnarray}%
where $a_{2}=\tilde{g}(\tilde{\nabla}_{\xi }e_{1},e_{2}),$ $b_{2}=\tilde{g}(%
\tilde{\nabla}_{e_{1}}e_{2},e_{1})$ and $\tilde{b}_{2}=-\frac{1}{2}$ $\tilde{%
\sigma}(e_{1})=-\frac{1}{2}\tilde{S}(\xi ,e_{1})$ .
\end{lemma}

\begin{proof}
Replacing $X$ \ by $e_{1}$ and $Y$ by $e_{2}$ in equation $\tilde{\nabla}\xi
=-\tilde{\varphi}+\tilde{\varphi}\tilde{h}$, we have $iii),vi).$

For the proof of $viii)$ we have%
\begin{eqnarray*}
\tilde{\nabla}_{\xi }e_{2} &=&\tilde{g}(\tilde{\nabla}_{\xi
}e_{2},e_{2})e_{1}+\tilde{g}(\tilde{\nabla}_{\xi }e_{2},e_{1})e_{2}+\tilde{g}%
(\tilde{\nabla}_{\xi }e_{2},\xi )\xi \\
&=&-\tilde{g}(e_{2},\tilde{\nabla}_{\xi }e_{1})e_{2}.
\end{eqnarray*}

If the function $a_{2}$ is defined as $\tilde{g}(\tilde{\nabla}_{\xi
}e_{1},e_{2})$ then $\tilde{\nabla}_{_{\xi }}e_{2}=-a_{2}e_{2}$. The proofs
of other covariant derivative equalities are similar to $ii).$

Putting $X=e_{1},$ $Y=e_{2}$ and \ $Z=\xi $ in the equation (\ref{THREE DIM
CURVATURE}), we have%
\begin{equation}
\tilde{R}(e_{1},e_{2})\xi =-\tilde{\sigma}(e_{1})e_{2}+\tilde{\sigma}%
(e_{2})e_{1}.  \label{iremm2}
\end{equation}%
On the other hand, by using (\ref{CURVATURE 4}), we get 
\begin{eqnarray}
\tilde{R}(e_{1},e_{2})\xi &=&(\tilde{\nabla}_{e_{1}}\tilde{\varphi}\tilde{h}%
)e_{2}-(\tilde{\nabla}_{e_{2}}\tilde{\varphi}\tilde{h})e_{1}  \notag \\
&=&2\tilde{b}_{2}e_{2}.  \label{iremm3}
\end{eqnarray}%
Comparing (\ref{iremm3}) with (\ref{iremm2}), we obtain%
\begin{equation}
\tilde{\sigma}(e_{1})=-2\tilde{b}_{2},~\ \tilde{\sigma}(e_{2})=0=\tilde{S}%
(\xi ,e_{2}).\   \label{iremm3a}
\end{equation}%
Hence, the function $\tilde{b}_{2}$ is obtained from the last equation.
\end{proof}

Next, we derive a useful formula for $\tilde{\nabla}_{\xi }\tilde{h}.$

\begin{proposition}
\label{p2}Let $(M,\tilde{\varphi},\xi ,\eta ,\tilde{g})$ be a $3$%
-dimensional paracontact metric manifold with $\tilde{h}$ of $\mathfrak{h}%
_{2}$ type. Then we have%
\begin{equation}
\tilde{\nabla}_{\xi }\tilde{h}=2a_{2}\tilde{\varphi}\tilde{h},
\label{iremm4}
\end{equation}%
on $\mathcal{U}$ .
\end{proposition}

\begin{proof}
Using (\ref{iremm0}), we get%
\begin{eqnarray*}
(\tilde{\nabla}_{\xi }\tilde{h})\xi &=&0=(2a_{2}\tilde{\varphi}\tilde{h})\xi
, \\
(\tilde{\nabla}_{\xi }\tilde{h})e_{1} &=&-2a_{2}e_{2}=(2a_{2}\tilde{\varphi}%
\tilde{h})e_{1}, \\
(\tilde{\nabla}_{\xi }\tilde{h})e_{2} &=&0=(2a_{2}\tilde{\varphi}\tilde{h}%
)e_{2}.
\end{eqnarray*}
\end{proof}

Using (\ref{RICCI ZETA}), we have $\tilde{S}(\xi ,\xi )=-2.$ After
calculating $\tilde{h}^{2}-\tilde{\varphi}^{2}$ with respect to the basis
components, we get%
\begin{equation}
\tilde{h}^{2}\xi -\tilde{\varphi}^{2}\xi =\frac{\tilde{S}(\xi ,\xi )}{2}%
\tilde{\varphi}^{2}\xi =0,\text{ \ }\tilde{h}^{2}e_{1}-\tilde{\varphi}%
^{2}e_{1}=\frac{\tilde{S}(\xi ,\xi )}{2}\tilde{\varphi}^{2}e_{1},\text{ \ }%
\tilde{h}^{2}e_{2}-\tilde{\varphi}^{2}e_{2}=\frac{\tilde{S}(\xi ,\xi )}{2}%
\tilde{\varphi}^{2}e_{2}.  \label{iremm6}
\end{equation}%
So we have following

\begin{proposition}
\label{p3}Let $(M,\tilde{\varphi},\xi ,\eta ,\tilde{g})$ be a $3$%
-dimensional paracontact metric manifold. If $\tilde{h}$ is $\mathfrak{h}%
_{2} $ type then the following equation holds on $M.$%
\begin{equation}
\tilde{h}^{2}-\tilde{\varphi}^{2}=\frac{\tilde{S}(\xi ,\xi )}{2}\tilde{%
\varphi}^{2}.  \label{iremm5}
\end{equation}
\end{proposition}

\begin{lemma}
\label{p4}Let $(M,\tilde{\varphi},\xi ,\eta ,\tilde{g})$ be a $3$%
-dimensional paracontact metric manifold with $\tilde{h}$ of $\mathfrak{h}%
_{2}$ type. Then the Ricci operator $\tilde{Q}$ is given by%
\begin{equation}
\tilde{Q}=\ddot{a}I+\ddot{b}\eta \otimes \xi -\tilde{\varphi}(\tilde{\nabla}%
_{\xi }\tilde{h})+\tilde{\sigma}(\tilde{\varphi}^{2})\otimes \xi +\tilde{%
\sigma}(e_{1})\eta \otimes e_{2},  \label{iremm7}
\end{equation}%
where $\ddot{a}$ and $\ddot{b}$ are smooth functions defined by $\ddot{a}$ $%
=1+\frac{r}{2}$ and $\ddot{b}=-3-\frac{r}{2},$ respectively.
\end{lemma}

\begin{proof}
For $3$-dimensional case, we have%
\begin{equation*}
\tilde{l}X=\tilde{S}(\xi ,\xi )X-\tilde{S}(X,\xi )\xi +\tilde{Q}X-\eta (X)%
\tilde{Q}\xi -\frac{r}{2}(X-\eta (X)\xi ),
\end{equation*}%
where $\tilde{l}$ denotes the Jacobi operator and $X$ is a vector field.
Using (\ref{irem000}), the last equation implies%
\begin{equation}
\tilde{Q}X=-\tilde{\varphi}^{2}X+\tilde{h}^{2}X-\tilde{\varphi}(\tilde{\nabla%
}_{\xi }\tilde{h})X-\tilde{S}(\xi ,\xi )X+\tilde{S}(X,\xi )\xi +\eta (X)%
\tilde{Q}\xi +\frac{r}{2}(X-\eta (X)\xi ).  \label{QX1}
\end{equation}%
Using $\tilde{\varphi}^{2}=I-\eta \otimes \xi ,$ we get $\tilde{S}(X,\xi )=%
\tilde{S}(\tilde{\varphi}^{2}X,\xi )+\eta (X)\tilde{S}(\xi ,\xi )$. So (\ref%
{QX1}) becomes%
\begin{equation}
\tilde{Q}X=\frac{\tilde{S}(\xi ,\xi )}{2}\tilde{\varphi}^{2}X-\tilde{\varphi}%
(\tilde{\nabla}_{\xi }\tilde{h})X-\tilde{S}(\xi ,\xi )X+\tilde{S}(\tilde{%
\varphi}^{2}X,\xi )\xi +\eta (X)\tilde{S}(\xi ,\xi )\xi +\eta (X)\tilde{Q}%
\xi +\frac{r}{2}\tilde{\varphi}^{2}X.  \label{iremm8}
\end{equation}%
By the basis pseudo-orthonormal $\{e_{1},e_{2},\xi \}$ and (\ref{iremm3a}),
it follows that 
\begin{equation}
\tilde{Q}\xi =\tilde{\sigma}(e_{1})e_{2}+\tilde{S}(\xi ,\xi )\xi .
\label{iremm9}
\end{equation}%
Using (\ref{iremm9}) in (\ref{iremm8}), we have%
\begin{eqnarray}
\tilde{Q}X &=&\left( 1+\frac{r}{2}\right) X+\left( -3-\frac{r}{2}\right)
\eta (X)\xi  \label{iremm10} \\
&&-\tilde{\varphi}(\tilde{\nabla}_{\xi }\tilde{h})X+\tilde{\sigma}(\tilde{%
\varphi}^{2}X)\xi +\eta (X)\tilde{\sigma}(e_{1})e_{2},  \notag
\end{eqnarray}%
for arbitrary vector field $X$. Hence, proof comes from (\ref{iremm10}).
\end{proof}

A consequence of Lemma \ref{p4}, we can give the components of the Ricci
operator $\tilde{Q}$ by following,

\begin{eqnarray}
\tilde{Q}\xi &=&(\ddot{a}+\ddot{b})\xi +\tilde{\sigma}(e_{1})e_{2},  \notag
\\
\tilde{Q}e_{1} &=&\tilde{\sigma}(e_{1})\xi +\ddot{a}e_{1}-2ae_{2},
\label{iremm7a} \\
\tilde{Q}e_{2} &=&\ddot{a}e_{2}.  \notag
\end{eqnarray}

\textbf{The tensor }$\tilde{h}$\textbf{\ has the canonical form (III).} Let $%
(M,\tilde{\varphi},\xi ,\eta ,\tilde{g})$ be a $3$-dimensional paracontact
metric manifold and $p$ is a point of $M$. Then there exists a local
orthonormal $\tilde{\varphi}$-basis $\{\tilde{e},\tilde{\varphi}\tilde{e}%
,\xi \}$ in a neighborhood of $p$ where $-\tilde{g}(\tilde{e},\tilde{e})=%
\tilde{g}(\tilde{\varphi}\tilde{e},\tilde{\varphi}\tilde{e})=\tilde{g}(\xi
,\xi )=1.$ Now, let $U_{1}$ be the open subset of $M$ where $\tilde{h}\neq 0$
and let $U_{2}$ be the open subset of points $p\in M$ such that $\tilde{h}=0$
in a neighborhood of $p.$ $U_{1}\cup $ $U_{2}$ is an open subset of $M$. For
every $p\in U_{1}$ there exists an open neighborhood of $p$ such that $%
\tilde{h}\tilde{e}=\tilde{\lambda}\tilde{\varphi}\tilde{e},\tilde{h}\tilde{%
\varphi}\tilde{e}=-\tilde{\lambda}\tilde{e}$ and $\tilde{h}\xi =0$ where $%
\tilde{\lambda}$ is a non-vanishing smooth function. Since $tr\tilde{h}=0,\ $%
the matrix form of $\tilde{h}$ is$\ $given $\ $by%
\begin{equation}
\tilde{h}=\left( 
\begin{array}{ccc}
0 & -\tilde{\lambda} & 0 \\ 
\tilde{\lambda} & 0 & 0 \\ 
0 & 0 & 0%
\end{array}%
\right)  \label{A3}
\end{equation}%
with respect to local orthonormal basis $\{\tilde{e},\tilde{\varphi}\tilde{e}%
,\xi \}.$ In this case, we say that $\tilde{h}$ is of $\mathfrak{h}_{3}$ 
\textit{type}.

\begin{lemma}
\label{lem3}Let $(M,\tilde{\varphi},\xi ,\eta ,\tilde{g})$ be a $3$%
-dimensional paracontact metric manifold with $\tilde{h}$ of $\mathfrak{h}%
_{3}$ type. Then for the covariant derivative on $U_{1}$ the following
equations are valid 
\begin{eqnarray}
i)\text{ }\tilde{\nabla}_{\tilde{e}}\tilde{e} &=&a_{3}\tilde{\varphi}\tilde{e%
}+\tilde{\lambda}\xi ,\text{ \ }ii)\text{ }\tilde{\nabla}_{\tilde{e}}\tilde{%
\varphi}\tilde{e}=a_{3}\tilde{e}+\xi ,\text{ \ }iii)\text{ }\tilde{\nabla}_{%
\tilde{e}}\xi =-\tilde{\varphi}\tilde{e}+\tilde{\lambda}\tilde{e},  \notag \\
iv)\text{ }\tilde{\nabla}_{\tilde{\varphi}\tilde{e}}\tilde{e} &=&b_{3}\tilde{%
\varphi}\tilde{e}-\xi ,\text{ \ }v)\text{ }\tilde{\nabla}_{\tilde{\varphi}%
\tilde{e}}\tilde{\varphi}\tilde{e}=b_{3}\tilde{e}+\tilde{\lambda}\xi ,\text{
\ }vi)\text{ }\tilde{\nabla}_{\tilde{\varphi}\tilde{e}}\xi =-\tilde{e}-%
\tilde{\lambda}\tilde{\varphi}\tilde{e},  \notag \\
vii)\text{ }\tilde{\nabla}_{\xi }\tilde{e} &=&\tilde{b}_{3}\tilde{\varphi}%
\tilde{e},\text{ \ }viii)\text{ }\tilde{\nabla}_{\xi }\tilde{\varphi}\tilde{e%
}=\tilde{b}_{3}\tilde{e},\text{ \ \ }  \notag \\
ix)\text{ }[\tilde{e},\xi ] &=&\tilde{\lambda}\tilde{e}-(1+\tilde{b}_{3})%
\tilde{\varphi}\tilde{e},\text{ \ }x)\text{ }[\tilde{\varphi}\tilde{e},\xi
]=-(1+\tilde{b}_{3})\tilde{e}-\tilde{\lambda}\tilde{\varphi}\tilde{e},\text{
\ }xi)\text{ }[\tilde{e},\tilde{\varphi}\tilde{e}]=a_{3}\tilde{e}-b_{3}%
\tilde{\varphi}\tilde{e}+2\xi .  \label{iremmm0}
\end{eqnarray}%
where $a_{3}$, $b_{3}$ and $\tilde{b}_{3}$ are defined by%
\begin{equation*}
a_{3}=-\frac{1}{2\tilde{\lambda}}\left[ \sigma (\tilde{\varphi}\tilde{e})+(%
\tilde{\varphi}\tilde{e})(\tilde{\lambda})\right] ,\text{ \ }\tilde{\sigma}(%
\tilde{e})=S(\xi ,\tilde{e}),\text{\ }
\end{equation*}%
\begin{equation*}
b_{3}=\frac{1}{2\tilde{\lambda}}\left[ \sigma (\tilde{e})-\tilde{e}(\tilde{%
\lambda})\right] \text{\ },\text{ \ }\tilde{\sigma}(\tilde{\varphi}\tilde{e}%
)=S(\xi ,\tilde{\varphi}\tilde{e}),\text{\ \ }
\end{equation*}%
\begin{equation*}
\tilde{b}_{3}=\tilde{g}(\tilde{\nabla}_{\xi }\tilde{e},\tilde{\varphi}\tilde{%
e})
\end{equation*}%
respectively.
\end{lemma}

\begin{proof}
Replacing $X$ $\ $by $\tilde{e}$ and $Y$ $\ $by $\tilde{\varphi}\tilde{e}$
in equation $\tilde{\nabla}\xi =-\tilde{\varphi}+\tilde{\varphi}\tilde{h}$,
we have $iii),vi).$

For the proof of $viii)$ we \ have%
\begin{eqnarray*}
\tilde{\nabla}_{\xi }\tilde{\varphi}\tilde{e} &=&-\tilde{g}(\tilde{\nabla}%
_{\xi }\tilde{\varphi}\tilde{e},\tilde{e})\tilde{e}+\tilde{g}(\tilde{\nabla}%
_{\xi }\tilde{\varphi}\tilde{e},\tilde{\varphi}\tilde{e})\tilde{\varphi}%
\tilde{e}+\tilde{g}(\tilde{\nabla}_{\xi }\tilde{\varphi}\tilde{e},\xi )\xi \\
&=&\tilde{g}(\tilde{\varphi}\tilde{e},\tilde{\nabla}_{\xi }\tilde{e})\tilde{e%
},
\end{eqnarray*}

where $\tilde{b}_{3}$ $=\tilde{g}(\tilde{\nabla}_{\xi }\tilde{e},\tilde{%
\varphi}\tilde{e}).$ So we obtain $\tilde{\nabla}_{\xi }\tilde{\varphi}%
\tilde{e}=\tilde{b}_{3}\tilde{e}$ . The proofs of other covariant derivative
equalities are similar to $viii).$

Putting $X=\tilde{e},$ $Y=\tilde{\varphi}\tilde{e},$ \ $Z=\xi $ in the
equation (\ref{THREE DIM CURVATURE}), we have%
\begin{equation*}
\tilde{R}(\tilde{e},\tilde{\varphi}\tilde{e})\xi =-\tilde{g}(\tilde{Q}\tilde{%
e},\xi )\tilde{\varphi}\tilde{e}+\tilde{g}(\tilde{Q}\tilde{\varphi}\tilde{e}%
,\xi )\tilde{e}.
\end{equation*}%
Since $\tilde{\sigma}(X)$ is defined as $\tilde{g}(\tilde{Q}\xi ,X),$ we
have 
\begin{equation}
\tilde{R}(\tilde{e},\tilde{\varphi}\tilde{e})\xi =-\tilde{\sigma}(\tilde{e})%
\tilde{\varphi}\tilde{e}+\tilde{\sigma}(\tilde{\varphi}\tilde{e})\tilde{e}.
\label{iremmm2}
\end{equation}%
On the other hand, by using (\ref{CURVATURE 4}), we have 
\begin{eqnarray}
\tilde{R}(\tilde{e},\tilde{\varphi}\tilde{e})\xi &=&(\tilde{\nabla}_{\tilde{e%
}}\tilde{\varphi}\tilde{h})\tilde{\varphi}\tilde{e}-(\tilde{\nabla}_{\tilde{%
\varphi}\tilde{e}}\tilde{\varphi}\tilde{h})\tilde{e}  \notag \\
&=&(-2a_{3}\tilde{\lambda}-(\tilde{\varphi}\tilde{e})(\tilde{\lambda}))%
\tilde{e}+(-2b_{3}\tilde{\lambda}-\tilde{e}(\tilde{\lambda}))\tilde{\varphi}%
\tilde{e}.  \label{iremmm3}
\end{eqnarray}%
Comparing (\ref{iremmm3}) with (\ref{iremmm2}), we get%
\begin{equation*}
\tilde{\sigma}(\tilde{e})=\tilde{e}(\tilde{\lambda})+2b_{3}\tilde{\lambda}%
,~\ \tilde{\sigma}(\tilde{\varphi}\tilde{e})=-(\tilde{\varphi}\tilde{e})(%
\tilde{\lambda})-2a_{3}\tilde{\lambda}.\ 
\end{equation*}%
Hence, the functions $a_{3}$ and $b_{3}$ are obtained from the last equation.
\end{proof}

Next, we derive a useful formula for $\tilde{\nabla}_{\xi }\tilde{h}.$

\begin{proposition}
\label{de3}Let $(M,\tilde{\varphi},\xi ,\eta ,\tilde{g})$ be a $3$%
-dimensional paracontact metric manifold with $\tilde{h}$ of $\mathfrak{h}%
_{3}$ type. So, on $U_{1}$ we have%
\begin{equation}
\tilde{\nabla}_{\xi }\tilde{h}=-2\tilde{b}_{3}\tilde{h}\tilde{\varphi}+\xi (%
\tilde{\lambda})s,  \label{iremmm4}
\end{equation}%
where $s$ is the $(1,1)$-type tensor defined by $s\xi =0,$ $s\tilde{e}=%
\tilde{\varphi}\tilde{e},$ \ $s\tilde{\varphi}\tilde{e}=-\tilde{e}.$
\end{proposition}

\begin{proof}
Using (\ref{iremmm0}), we get%
\begin{eqnarray*}
(\tilde{\nabla}_{\xi }\tilde{h})\xi &=&0=(-2\tilde{b}_{3}\tilde{h}\tilde{%
\varphi}+\xi (\tilde{\lambda})s)\xi , \\
(\tilde{\nabla}_{\xi }\tilde{h})\tilde{e} &=&2\tilde{b}_{3}\tilde{\lambda}%
\tilde{e}+\xi (\tilde{\lambda})\tilde{\varphi}\tilde{e}=(-2\tilde{b}_{3}%
\tilde{h}\tilde{\varphi}+\xi (\tilde{\lambda})s)\tilde{e}, \\
(\tilde{\nabla}_{\xi }\tilde{h})\tilde{\varphi}\tilde{e} &=&-2\tilde{b}_{3}%
\tilde{\lambda}\tilde{\varphi}\tilde{e}-\xi (\tilde{\lambda})\tilde{e}=(-2%
\tilde{b}_{3}\tilde{h}\tilde{\varphi}+\xi (\tilde{\lambda})s)\tilde{\varphi}%
\tilde{e}.
\end{eqnarray*}%
The last equations complete the proof of (\ref{iremmm4}).
\end{proof}

\begin{proposition}
\label{pe3}Let $(M,\tilde{\varphi},\xi ,\eta ,\tilde{g})$ be a $3$%
-dimensional paracontact metric manifold with $\tilde{h}$ of $\mathfrak{h}%
_{3}$ type. Then the following equation holds on $M.$%
\begin{equation}
\tilde{h}^{2}-\tilde{\varphi}^{2}=\frac{\tilde{S}(\xi ,\xi )}{2}\tilde{%
\varphi}^{2}.  \label{iremmm5}
\end{equation}
\end{proposition}

\begin{proof}
Using (\ref{RICCI ZETA}), we have $\tilde{S}(\xi ,\xi )=2(1+\tilde{\lambda}%
^{2}).$ After calculating $\tilde{h}^{2}-\tilde{\varphi}^{2}$ with respect
to the basis components, we get%
\begin{equation}
\tilde{h}^{2}\xi -\tilde{\varphi}^{2}\xi =\frac{\tilde{S}(\xi ,\xi )}{2}%
\tilde{\varphi}^{2}\xi =0,\text{ \ }\tilde{h}^{2}\tilde{e}-\tilde{\varphi}%
^{2}\tilde{e}=\frac{\tilde{S}(\xi ,\xi )}{2}\tilde{\varphi}^{2}\tilde{e},%
\text{ \ }\tilde{h}^{2}\tilde{\varphi}\tilde{e}-\tilde{\varphi}^{3}\tilde{e}=%
\frac{\tilde{S}(\xi ,\xi )}{2}\tilde{\varphi}^{2}\tilde{\varphi}\tilde{e}.
\label{iremmm6}
\end{equation}%
(\ref{iremmm6}) completes the proof of (\ref{iremmm5}).
\end{proof}

\begin{lemma}
\label{le3}Let $(M,\tilde{\varphi},\xi ,\eta ,\tilde{g})$ be a $3$%
-dimensional paracontact metric manifold with $\tilde{h}$ of $\mathfrak{h}%
_{3}$ type. Then the Ricci operator $\tilde{Q}$ is given by%
\begin{equation}
\tilde{Q}=\bar{a}\text{ }I+\bar{b}\eta \otimes \xi -\tilde{\varphi}(\tilde{%
\nabla}_{\xi }\tilde{h})+\tilde{\sigma}(\tilde{\varphi}^{2})\otimes \xi -%
\tilde{\sigma}(\tilde{e})\eta \otimes \tilde{e}+\tilde{\sigma}(\tilde{\varphi%
}\tilde{e})\eta \otimes \tilde{\varphi}\tilde{e},  \label{iremmm7}
\end{equation}%
where $\bar{a}$ and $\bar{b}$ are smooth functions defined by $\bar{a}$ $=1+%
\tilde{\lambda}^{2}+\frac{r}{2}$ and $\bar{b}=-3(\tilde{\lambda}^{2}+1)-%
\frac{r}{2},$ respectively. Moreover the components of the Ricci operator $%
\tilde{Q}$ are given by%
\begin{eqnarray}
\tilde{Q}\xi &=&(\bar{a}+\bar{b})\xi -\tilde{\sigma}(\tilde{e})\tilde{e}+%
\tilde{\sigma}(\tilde{\varphi}\tilde{e})\tilde{\varphi}\tilde{e},  \notag \\
\tilde{Q}\tilde{e} &=&\tilde{\sigma}(\tilde{e})\xi +(\bar{a}+\xi (\tilde{%
\lambda}))\tilde{e}-2\tilde{b}_{3}\tilde{\lambda}\tilde{\varphi}\tilde{e},
\label{iremmm7a} \\
\tilde{Q}\tilde{\varphi}\tilde{e} &=&\tilde{\sigma}(\tilde{\varphi}\tilde{e}%
)\xi +2\tilde{b}_{3}\tilde{\lambda}\tilde{e}+(\bar{a}+\xi (\tilde{\lambda}))%
\tilde{\varphi}\tilde{e}.  \notag
\end{eqnarray}
\end{lemma}

\begin{proof}
By (\ref{THREE DIM CURVATURE}), we have%
\begin{equation*}
\tilde{R}(X,\xi )\xi =\tilde{S}(\xi ,\xi )X-\tilde{S}(X,\xi )\xi +\tilde{Q}%
X-\eta (X)\tilde{Q}\xi -\frac{r}{2}(X-\eta (X)\xi ),
\end{equation*}%
for any vector field $X.$ Using (\ref{irem000}), the last equation implies%
\begin{equation}
\tilde{Q}X=-\tilde{\varphi}^{2}X+\tilde{h}^{2}X-\tilde{\varphi}(\tilde{\nabla%
}_{\xi }\tilde{h})X-\tilde{S}(\xi ,\xi )X+\tilde{S}(X,\xi )\xi +\eta (X)%
\tilde{Q}\xi +\frac{r}{2}(X-\eta (X)\xi ).  \label{QX}
\end{equation}%
By writing $\tilde{S}(X,\xi )=\tilde{S}(\tilde{\varphi}^{2}X,\xi )+\eta (X)%
\tilde{S}(\xi ,\xi )$ in (\ref{QX}), we obtain 
\begin{equation}
\tilde{Q}X=\frac{\tilde{S}(\xi ,\xi )}{2}\tilde{\varphi}^{2}X-\tilde{\varphi}%
(\tilde{\nabla}_{\xi }\tilde{h})X-\tilde{S}(\xi ,\xi )X+\tilde{S}(\tilde{%
\varphi}^{2}X,\xi )\xi +\eta (X)\tilde{S}(\xi ,\xi )\xi +\eta (X)\tilde{Q}%
\xi +\frac{r}{2}\tilde{\varphi}^{2}X.  \label{iremmm8}
\end{equation}%
We know that the Ricci tensor $\tilde{S}$ with respect to the orthonormal
basis $\{\tilde{e},\tilde{\varphi}\tilde{e},\xi \}$ is given by 
\begin{equation}
\tilde{Q}\xi =-\tilde{\sigma}(\tilde{e})\tilde{e}+\tilde{\sigma}(\tilde{%
\varphi}\tilde{e})\tilde{\varphi}\tilde{e}+\tilde{S}(\xi ,\xi )\xi .
\label{iremmm9}
\end{equation}%
Using (\ref{iremmm9}) in (\ref{iremmm8}), we have%
\begin{eqnarray}
\tilde{Q}X &=&\left( 1+\tilde{\lambda}^{2}+\frac{r}{2}\right) X+\left( -3(%
\tilde{\lambda}^{2}+1)-\frac{r}{2}\right) \eta (X)\xi  \label{iremmm10} \\
&&-\tilde{\varphi}(\tilde{\nabla}_{\xi }\tilde{h})X+\tilde{\sigma}(\tilde{%
\varphi}^{2}X)\xi -\eta (X)\tilde{\sigma}(\tilde{e})\tilde{e}+\eta (X)\tilde{%
\sigma}(\tilde{\varphi}\tilde{e})\tilde{\varphi}\tilde{e},  \notag
\end{eqnarray}%
for arbitrary vector field $X$. This ends the proof.
\end{proof}

\textbf{The tensor }$\tilde{h}$\textbf{\ has the canonical form (IV).} Let $%
(M,\tilde{\varphi},\xi ,\eta ,\tilde{g})$ be a $3$-dimensional paracontact
metric manifold and $p$ is a point of $M$. Then there exists a local
pseudo-orthonormal basis $\{e_{1},e_{2},e_{3}\}$ in a neighborhood of $p$
where $\tilde{g}(e_{1},e_{1})=\tilde{g}(e_{2},e_{2})=\tilde{g}(e_{1},e_{3})=%
\tilde{g}(e_{2},e_{3})=0$ and $\tilde{g}(e_{1},e_{2})=\tilde{g}%
(e_{3},e_{3})=1.$ Since the tensor $\tilde{h}$ has canonical form (IV) (with
respect to a pseudo-orthonormal basis $\{e_{1},e_{2},e_{3}\})$ then $\tilde{h%
}e_{1}=\tilde{\lambda}e_{1}+e_{3},$ \ $\tilde{h}e_{2}=\tilde{\lambda}e_{2}$
and $\tilde{h}e_{3}=e_{2}+\tilde{\lambda}e_{3}.$ Since $0=tr\tilde{h}=\tilde{%
g}(\tilde{h}e_{1},e_{2})+\tilde{g}(\tilde{h}e_{2},e_{1})+\tilde{g}(\tilde{h}%
e_{3},e_{3})=3\tilde{\lambda}$, then $\tilde{\lambda}=0.$ We write $\xi =%
\tilde{g}(\xi ,e_{2})e_{1}+\tilde{g}(\xi ,e_{1})e_{2}+\tilde{g}(\xi
,e_{3})e_{3}$ respect to the pseudo-orthonormal basis $\{e_{1},e_{2},e_{3}\}$%
. Since $\tilde{h}\xi =0,$ we have $0=\tilde{g}(\xi ,e_{2})e_{3}+\tilde{g}%
(\xi ,e_{3})e_{2}.$ Hence we get $\xi =\tilde{g}(\xi ,e_{1})e_{2}$ which
leads to a contradiction with $\tilde{g}(\xi ,\xi )=1$. Thus, this case does
not occur.

\textbf{The proof of Theorem \ref{DESTAR}:}

According to a suitable (pseudo)-orthonormal basis, $\tilde{h}$ can be put
into one of the three forms. So we will need three cases.

\textit{Case 1:} \ Firstly, we suppose that $\tilde{h}$ is $\mathfrak{h}_{1}$
type.

From (\ref{NAMBLASTAR}) we have%
\begin{eqnarray*}
\tilde{\nabla}^{\ast }\tilde{\nabla}\xi &=&-\tilde{\nabla}_{\tilde{e}}\tilde{%
\nabla}_{\tilde{e}}\xi +\tilde{\nabla}_{\tilde{\nabla}_{\tilde{e}}\tilde{e}%
}\xi +\tilde{\nabla}_{\tilde{\varphi}\tilde{e}}\tilde{\nabla}_{\tilde{\varphi%
}\tilde{e}}\xi -\tilde{\nabla}_{\tilde{\nabla}_{\tilde{\varphi}\tilde{e}}%
\tilde{\varphi}\tilde{e}}\xi \\
&=&-\tilde{\sigma}(\tilde{e})\tilde{e}+\tilde{\sigma}(\tilde{\varphi}\tilde{e%
})\tilde{\varphi}\tilde{e}+2(\tilde{\lambda}^{2}+1)\xi .
\end{eqnarray*}

By using (\ref{NAMBLASTAR4}) and (\ref{irem7a}) we obtain $\tilde{\sigma}(%
\tilde{e})=\tilde{\sigma}(\tilde{\varphi}\tilde{e})=0$ and $\tilde{Q}\xi =2(%
\tilde{\lambda}^{2}-1)\xi .$

\textit{Case 2:} \ Secondly, we assume that $\tilde{h}$ is $\mathfrak{h}_{2}$
type.

We construct an orthonormal basis $\{\tilde{e},\tilde{\varphi}\tilde{e},\xi
\}$ from the pseudo-orthonormal basis $\{e_{1},e_{2},\xi \}$ such that 
\begin{equation}
\tilde{e}=\frac{e_{1}-e_{2}}{\sqrt{2}},\text{ \ \ }\tilde{\varphi}\tilde{e}=%
\frac{e_{1}+e_{2}}{\sqrt{2}}\text{, \ }\tilde{g}(\tilde{e},\tilde{e})=-1%
\text{ and }\tilde{g}(\tilde{\varphi}\tilde{e},\tilde{\varphi}\tilde{e})=1%
\text{.}  \label{cm1}
\end{equation}%
\ Then $\tilde{h}$ with respect to this new basis takes the form,%
\begin{equation}
\tilde{h}\tilde{e}=\tilde{h}\tilde{\varphi}\tilde{e}=\frac{1}{2}(-\tilde{e}+%
\tilde{\varphi}\tilde{e}).  \label{cm2}
\end{equation}%
By Lemma \ref{de2} we have%
\begin{eqnarray}
\tilde{\nabla}_{\tilde{e}}\tilde{e} &=&\frac{1}{\sqrt{2}}(-b_{2}-\frac{1}{2}%
\tilde{\sigma}(e_{1}))\tilde{\varphi}\tilde{e}+\frac{1}{2}\xi ,\text{ }%
\tilde{\nabla}_{\tilde{e}}\tilde{\varphi}\tilde{e}=-\frac{1}{\sqrt{2}}(b_{2}+%
\frac{1}{2}\tilde{\sigma}(e_{1}))\tilde{e}+\frac{3}{2}\xi ,  \label{irem3} \\
\text{ }\tilde{\nabla}_{\tilde{e}}\xi &=&\frac{\tilde{e}-3\tilde{\varphi}%
\tilde{e}}{2}\text{, \ \ }\tilde{\nabla}_{\tilde{\varphi}\tilde{e}}\tilde{e}%
=-\frac{1}{\sqrt{2}}(b_{2}-\frac{1}{2}\tilde{\sigma}(e_{1}))\tilde{\varphi}%
\tilde{e}-\frac{1}{2}\xi ,  \notag \\
\text{ \ \ }\tilde{\nabla}_{\tilde{\varphi}\tilde{e}}\tilde{\varphi}\tilde{e}
&=&-\frac{1}{\sqrt{2}}(b_{2}-\frac{1}{2}\tilde{\sigma}(e_{1}))\tilde{e}+%
\frac{1}{2}\xi ,\text{ }\tilde{\nabla}_{\tilde{\varphi}\tilde{e}}\xi =\frac{-%
\tilde{e}-\tilde{\varphi}\tilde{e}}{2}\text{\ \ .}  \notag
\end{eqnarray}%
Using (\ref{NAMBLASTAR}) and (\ref{irem3}) we get%
\begin{eqnarray}
\tilde{\nabla}^{\ast }\tilde{\nabla}\xi &=&-\tilde{\nabla}_{\tilde{e}}\tilde{%
\nabla}_{\tilde{e}}\xi +\tilde{\nabla}_{\tilde{\nabla}_{\tilde{e}}\tilde{e}%
}\xi +\tilde{\nabla}_{\tilde{\varphi}\tilde{e}}\tilde{\nabla}_{\tilde{\varphi%
}\tilde{e}}\xi -\tilde{\nabla}_{\tilde{\nabla}_{\tilde{\varphi}\tilde{e}%
\tilde{\varphi}}\tilde{e}}\xi  \label{irem laplace} \\
&=&\frac{\tilde{\sigma}(e_{1})}{\sqrt{2}}(\tilde{e}-\tilde{\varphi}\tilde{e}%
)+2\xi .  \notag
\end{eqnarray}%
From (\ref{NAMBLASTAR4}) and (\ref{irem laplace}), we obtain $\tilde{\sigma}%
(e_{1})=0.$ By help of (\ref{iremm7a}), $\xi $ is an eigenvector of the
Ricci operator.

\textit{Case 3:} Finally, let $\tilde{h}$ \ be $\mathfrak{h}_{3}$ type.

Again using (\ref{NAMBLASTAR}) we obtain%
\begin{eqnarray*}
\tilde{\nabla}^{\ast }\tilde{\nabla}\xi &=&-\tilde{\nabla}_{\tilde{e}}\tilde{%
\nabla}_{\tilde{e}}\xi +\tilde{\nabla}_{\tilde{\nabla}_{\tilde{e}}\tilde{e}%
}\xi +\tilde{\nabla}_{\tilde{\varphi}\tilde{e}}\tilde{\nabla}_{\tilde{\varphi%
}\tilde{e}}\xi -\tilde{\nabla}_{\tilde{\nabla}_{\tilde{\varphi}\tilde{e}}%
\tilde{\varphi}\tilde{e}}\xi \\
&=&-\tilde{\sigma}(\tilde{e})\tilde{e}+\tilde{\sigma}(\tilde{\varphi}\tilde{e%
})\tilde{\varphi}\tilde{e}+2(1-\tilde{\lambda}^{2})\xi .
\end{eqnarray*}%
From (\ref{NAMBLASTAR4}), we obtain $\tilde{\sigma}(\tilde{e})=\tilde{\sigma}%
(\tilde{\varphi}\tilde{e})=0.$ By (\ref{iremmm7a}) we have $\tilde{Q}\xi
=-2(1+\tilde{\lambda}^{2})\xi .$

This completes the proof.

\textbf{The proof of Theorem \ref{k mu vu}:}

We give proof of the theorem for three cases respect to chosen
(pseudo)-orthonormal basis.

\textit{Case 1: }We assume that $\tilde{h}$ is $\mathfrak{h}_{1}$ type.

Since $\xi $ is a harmonic vector field, $\xi $ is an eigenvector of $\tilde{%
Q}.$ Hence we obtain that $\tilde{\sigma}=0.$ Putting $s=\frac{1}{\tilde{%
\lambda}}\tilde{h}$ in (\ref{irem7}) we have%
\begin{equation}
\tilde{Q}=a_{1}I+b_{1}\eta \otimes \xi -2b\tilde{h}-\frac{\xi (\tilde{\lambda%
})}{\tilde{\lambda}}\tilde{\varphi}\tilde{h}.  \label{irem11}
\end{equation}%
Setting $Z=\xi $ in (\ref{THREE DIM CURVATURE}) and using (\ref{irem11}), we
obtain%
\begin{equation*}
\tilde{R}(X,Y)\xi =(\tilde{\lambda}^{2}-1)(\eta (Y)X-\eta (X)Y)-2b(\eta (Y)%
\tilde{h}X-\eta (X)\tilde{h}Y)-\frac{\xi (\tilde{\lambda})}{\tilde{\lambda}}%
(\eta (Y)\tilde{\varphi}\tilde{h}X-\eta (X)\tilde{\varphi}\tilde{h}Y),
\end{equation*}%
where the functions $\tilde{\kappa},$ $\tilde{\mu}$ and $\tilde{\nu}$
defined by $\tilde{\kappa}=\frac{\tilde{S}(\xi ,\xi )}{2},$ $\tilde{\mu}%
=-2b, $ $\tilde{\nu}=-\frac{\xi (\tilde{\lambda})}{\tilde{\lambda}},$
respectively. So, it is obvious that for this type $\tilde{\kappa}>-1$.
Moreover, using (\ref{irem11}), we have $\tilde{Q}\tilde{\varphi}-\tilde{%
\varphi}\tilde{Q}=2\tilde{\mu}\tilde{h}\tilde{\varphi}-2\tilde{\nu}\tilde{h}%
. $

\textit{Case 2: }Secondly, let $\tilde{h}$ be $\mathfrak{h}_{2}$ type.

Putting $\tilde{\sigma}=0$ in (\ref{iremm7}) we get%
\begin{equation}
\tilde{Q}=\ddot{a}I+\ddot{b}\eta \otimes \xi -2a_{2}\tilde{h},
\label{iremm11}
\end{equation}%
which yields%
\begin{equation}
\tilde{Q}\xi =\tilde{S}(\xi ,\xi )\xi ,  \label{iremm12}
\end{equation}%
for any vector fields on $M.$ Putting $\xi $ instead of $Z$ in (\ref{THREE
DIM CURVATURE}) we obtain%
\begin{eqnarray}
\tilde{R}(X,Y)\xi &=&-\tilde{S}(X,\xi )+\tilde{S}(Y,\xi )-\eta (X)\tilde{Q}Y
\label{iremm13} \\
&&+\eta (Y)\tilde{Q}X+\frac{r}{2}(\eta (X)Y-\eta (Y)X),  \notag
\end{eqnarray}%
for any vector field $X.$ Using (\ref{iremm11}) and (\ref{iremm12}) in (\ref%
{iremm13}), we obtain%
\begin{equation*}
\tilde{R}(X,Y)\xi =-(\eta (Y)X-\eta (X)Y)-2a_{2}(\eta (Y)\tilde{h}X-\eta (X)%
\tilde{h}Y),
\end{equation*}%
where the functions $\tilde{\kappa}$ and $\tilde{\mu}$ defined by $\tilde{%
\kappa}=\frac{\tilde{S}(\xi ,\xi )}{2},$ $\tilde{\mu}=-2a_{2},$
respectively. So, it is obvious that for this type $\tilde{\kappa}=-1$.
Furthermore, by (\ref{iremm11}), we have $\tilde{Q}\tilde{\varphi}-\tilde{%
\varphi}\tilde{Q}=2\tilde{\mu}\tilde{h}\tilde{\varphi}.$

\textit{Case 3:}\ Finally, we suppose that $\tilde{h}$ is $\mathfrak{h}_{3}$
type.

Since $M$ is a $H$-paracontact metric manifold we have $\tilde{\sigma}=0$.
Putting $s=\frac{1}{\tilde{\lambda}}\tilde{h}$ in (\ref{iremmm7}) we get%
\begin{equation}
\tilde{Q}=\bar{a}I+\bar{b}\eta \otimes \xi -2\tilde{b}_{3}\tilde{h}-\left( 
\frac{\xi (\tilde{\lambda})}{\tilde{\lambda}}\right) \tilde{\varphi}\tilde{h}%
,  \label{iremmm11}
\end{equation}%
which yields%
\begin{equation}
\tilde{Q}\xi =\tilde{S}(\xi ,\xi )\xi ,  \label{iremmm12}
\end{equation}%
for any vector fields on $M.$ Setting $\xi =$ $Z$ in (\ref{THREE DIM
CURVATURE}) we find%
\begin{eqnarray}
\tilde{R}(X,Y)\xi &=&-\tilde{S}(X,\xi )+\tilde{S}(Y,\xi )-\eta (X)\tilde{Q}Y
\label{iremmm13} \\
&&+\eta (Y)\tilde{Q}X+\frac{r}{2}(\eta (X)Y-\eta (Y)X),  \notag
\end{eqnarray}%
for any vector field $X.$ Using (\ref{iremmm11}) and (\ref{iremmm12}) in (%
\ref{iremmm13}), we get%
\begin{equation*}
\tilde{R}(X,Y)\xi =(-1-\tilde{\lambda}^{2})(\eta (Y)X-\eta (X)Y)-2\tilde{b}%
_{3}(\eta (Y)\tilde{h}X-\eta (X)\tilde{h}Y)-\frac{\xi (\tilde{\lambda})}{%
\tilde{\lambda}}(\eta (Y)\tilde{\varphi}\tilde{h}X-\eta (X)\tilde{\varphi}%
\tilde{h}Y),
\end{equation*}%
where the functions $\tilde{\kappa},$ $\tilde{\mu}$ and $\tilde{\nu}$
defined by $\tilde{\kappa}=\frac{\tilde{S}(\xi ,\xi )}{2},$ $\tilde{\mu}=-2%
\tilde{b}_{3},$ $\tilde{\nu}=-\frac{\xi (\tilde{\lambda})}{\tilde{\lambda}}$%
, respectively. So, it is obvious that for this type $\tilde{\kappa}<-1.$ By
help of (\ref{iremmm11}), we get $\tilde{Q}\tilde{\varphi}-\tilde{\varphi}%
\tilde{Q}=-2\tilde{\mu}\tilde{h}\tilde{\varphi}-2\tilde{\nu}\tilde{h}.$

Conversely, let $M$ is a paracontact metric $($ $\tilde{\kappa},\tilde{\mu},%
\tilde{\nu})$-manifold. Using Teorem \ref{DESTAR} and (\ref{Riczeta}), we
conclude that $\xi $ is harmonic vector field.

This completes the proof.

\textbf{Concluding Lemma:} \textit{\ Let }$(M,\tilde{\varphi},\xi ,\eta ,%
\tilde{g})$\textit{\ be a }$3$\textit{-dimensional paracontact metric
manifold. Then a canonical form of }$\tilde{h}$\textit{\ stays constant in
an open neighborhood of any point for }$\tilde{h}$\textit{. }

\textbf{Proof:} Let $U_{1}$ be open subset of $M$ where $\tilde{h}\neq 0$
and $p,q\in U_{1}$ with $p\neq q.$

\textit{Case 1}. Now we assume that $\tilde{h}_{p}$ has canonical form (I)
at $T_{p}M_{1}^{3}.$ In this case, there exists an orthonormal $\tilde{%
\varphi}$-basis $\{\tilde{e},\tilde{\varphi}\tilde{e},\xi \}$ such that 
\begin{equation}
\tilde{h}_{p}\tilde{e}=\tilde{\lambda}(p)\tilde{e}\text{, \ }\tilde{h}_{p}(%
\tilde{\varphi}\tilde{e})=-\text{\ }\tilde{\lambda}(p)\tilde{\varphi}\tilde{e%
}\text{, \ \ }\tilde{h}_{p}\xi =0,  \label{h1}
\end{equation}%
with $-\tilde{g}(\tilde{e},\tilde{e})=\tilde{g}(\tilde{\varphi}\tilde{e},%
\tilde{\varphi}\tilde{e})=\tilde{g}(\xi ,\xi )=1.$

In $T_{q}M_{1}^{3}$ we suppose that $\tilde{h}_{q}$ has canonical form (II).
By Lemma \ref{II} we can construct a pseudo-orthonormal basis $%
\{e_{1},e_{2},e_{3}\}$ in a neighborhood of $q$ such that $\tilde{h}%
_{q}e_{1}=e_{2},\tilde{h}_{q}e_{2}=0,\tilde{h}_{q}e_{3}=0$ and $\tilde{%
\varphi}e_{1}=e_{1},$ $\tilde{\varphi}e_{2}=-e_{2},\tilde{\varphi}e_{3}=0$
and also $\xi =e_{3}$ with $\tilde{g}(e_{1},e_{1})=\tilde{g}(e_{2},e_{2})=%
\tilde{g}(e_{1},e_{3})=\tilde{g}(e_{2},e_{3})=0$ and $\tilde{g}(e_{1},e_{2})=%
\tilde{g}(e_{3},e_{3})=1$. Putting 
\begin{equation}
\tilde{E}=\frac{e_{1}-e_{2}}{\sqrt{2}},\text{ \ \ \ }\tilde{\varphi}\tilde{E}%
=\frac{e_{1}+e_{2}}{\sqrt{2}}  \label{hxy}
\end{equation}%
we get an orthonormal basis such that $-\tilde{g}(\tilde{E},\tilde{E})=%
\tilde{g}(\tilde{\varphi}\tilde{E},\tilde{\varphi}\tilde{E})=\tilde{g}(\xi
,\xi )=1$ and also%
\begin{equation}
\tilde{h}_{q}\tilde{E}=\tilde{h}_{q}\tilde{\varphi}\tilde{E}=\frac{1}{2}(-%
\tilde{E}+\tilde{\varphi}\tilde{E}).  \label{hx}
\end{equation}%
So tangent spaces $T_{p}M$ $=span\{\tilde{e},\tilde{\varphi}\tilde{e},\xi \}$
and $T_{q}M=span\{\tilde{E},\tilde{\varphi}\tilde{E},\xi )$ have same
dimension and index. Thus there exist a linear isometry $F_{\ast }$ from $%
T_{p}M$ to $T_{q}M$ such that $F(p)=q$ and 
\begin{equation}
F_{\ast }(\tilde{e})=\tilde{E},\text{ \ \ }F_{\ast }(\tilde{\varphi}\tilde{e}%
)=\tilde{\varphi}\tilde{E},\text{ \ \ }F_{\ast }(\xi )=\xi .  \label{h2}
\end{equation}%
By (\ref{h1}) and (\ref{h2}) we have%
\begin{equation}
F_{\ast }(\tilde{h}_{p}\tilde{e})=\tilde{\lambda}(q)\tilde{E},\text{ }%
F_{\ast }(\tilde{h}_{p}\tilde{\varphi}\tilde{e})\text{\ }=-\tilde{\lambda}(q)%
\tilde{\varphi}\tilde{E}.  \label{h3}
\end{equation}%
Using (\ref{hx}) and (\ref{h3}) we obtain 
\begin{equation*}
0=tr\tilde{h}_{q}^{2}=2\tilde{\lambda}^{2}(q),
\end{equation*}%
which implies that $\tilde{\lambda}(q)=0$. Hence we contradict the fact that 
$q\in U_{1}$.

\textit{Case 2}. Again we assume that $\tilde{h}_{p}$ has canonical form (I)
at $T_{p}M_{1}^{3}.$ Suppose to contrary that $\tilde{h}_{q}$ has canonical
form (III) in $T_{q}M.$ One can construct a local orthonormal $\tilde{\varphi%
}$-basis $\{\tilde{f}_{1},\tilde{\varphi}\tilde{f}_{1},\xi \}$ in a
neighborhood of $q$ such that $-\tilde{g}(\tilde{f}_{1},\tilde{f}_{1})=%
\tilde{g}(\tilde{\varphi}\tilde{f}_{1},\tilde{\varphi}\tilde{f}_{1})=\tilde{g%
}(\xi ,\xi )=1,\tilde{h}_{q}(\tilde{f}_{1})=\tilde{\lambda}_{1}(q)\tilde{%
\varphi}\tilde{f}_{1},\tilde{h}_{q}(\tilde{\varphi}\tilde{f}_{1})=-\tilde{%
\lambda}_{1}(q)\tilde{f}_{1}$. Since tangent spaces $T_{p}M_{1}^{3}$ $=span\{%
\tilde{e},\tilde{\varphi}\tilde{e},\xi \}$ and $T_{q}M_{1}^{3}=span\{\tilde{f%
}_{1},\tilde{\varphi}\tilde{f}_{1},\xi )$ have same dimension and index, we
can construct a linear isometry $T_{\ast }$ from $T_{p}M_{1}^{3}$ to $%
T_{q}M_{1}^{3}$ such that $T(p)=q$ and 
\begin{equation}
T_{\ast }(\tilde{e})=\tilde{f}_{1},\text{ \ \ }T_{\ast }(\tilde{\varphi}%
\tilde{e})=\tilde{\varphi}\tilde{f}_{1},\text{ \ \ }T_{\ast }(\xi )=\xi .
\label{h4}
\end{equation}%
So we get 
\begin{equation}
T_{\ast }(\tilde{h}_{p}\tilde{e})=\tilde{\lambda}(q)\tilde{f}_{1},T_{\ast }(%
\tilde{h}_{p}\tilde{\varphi}\tilde{e})\ =-\tilde{\lambda}(q)\tilde{\varphi}%
\tilde{f}_{1}.  \label{h5}
\end{equation}%
From (\ref{h4}) and (\ref{h5}) we find%
\begin{equation*}
-2\tilde{\lambda}_{1}^{2}(q)=tr\tilde{h}_{q}^{2}=2\tilde{\lambda}^{2}(q)
\end{equation*}%
and this last equation gives $\tilde{\lambda}_{1}(q)=\tilde{\lambda}(q)=0$
which contradicts with $q\in U_{1}.$

\textit{Case 3}\textbf{.} Let us consider that $\tilde{h}_{p}$ has canonical
form (III) at $T_{p}M_{1}^{3}.$ In this case, there exist an orthonormal $%
\tilde{\varphi}$-basis $\{\tilde{e},\tilde{\varphi}\tilde{e},\xi \}$ such
that 
\begin{equation*}
\tilde{h}_{p}\tilde{e}=\tilde{\lambda}(p)\tilde{\varphi}\tilde{e}\text{, \ }%
\tilde{h}_{p}(\tilde{\varphi}\tilde{e})=-\tilde{\lambda}(p)\tilde{e}\text{,
\ \ }\tilde{h}_{p}\xi =0
\end{equation*}%
with $-\tilde{g}(\tilde{e},\tilde{e})=\tilde{g}(\tilde{\varphi}\tilde{e},%
\tilde{\varphi}\tilde{e})=\tilde{g}(\xi ,\xi )=1$. In $T_{q}M_{1}^{3}$ we
suppose that $\tilde{h}_{q}$ has canonical form (II). Similar arguments as
in Case 1, we obtain 
\begin{equation*}
0=tr\tilde{h}_{q}^{2}=2\tilde{\lambda}^{2}(q)
\end{equation*}%
which leads to a contradiction of chosen $q$. This completes proof of the
concluding lemma.

Now we will give some examples of $3$-dimensional $(\tilde{\kappa},$ $\tilde{%
\mu},\tilde{\nu})$-paracontact metric manifolds according to the cases $%
\tilde{\kappa}>-1$, $\tilde{\kappa}=-1$ and $\tilde{\kappa}<-1.$

\begin{example}
\label{E1}We consider the $3$-dimensional manifold 
\begin{equation*}
M=\{(x,y,z)\in R^{3}\mid 2y+z\neq 0,\text{ }z\neq 0\}
\end{equation*}%
and the vector fields%
\begin{equation*}
e_{1}=\text{ }\frac{\partial }{\partial x},\text{ \ \ }e_{2}=\frac{\partial 
}{\partial y},\text{ \ \ }e_{3}=(2y+z)\frac{\partial }{\partial x}-(2zx-%
\frac{1}{2z}y)\frac{\partial }{\partial y}+\frac{\partial }{\partial z}.
\end{equation*}%
The $1$-form $\eta =dx-(2y+z)dz$ \ defines a contact structure on $M$ with
characteristic vector field $\xi =\frac{\partial }{\partial x}$. We define
the structure tensor $\tilde{\varphi}\tilde{e}_{1}=0,$ $\tilde{\varphi}%
\tilde{e}_{2}=\tilde{e}_{3}$ and $\tilde{\varphi}\tilde{e}_{3}=\tilde{e}%
_{2}. $ Let $\tilde{g}$ be Lorentzian metric defined by $\tilde{g}%
(e_{1},e_{1})=-\tilde{g}(e_{2},e_{2})=\tilde{g}(e_{3},e_{3})=1$ and $\tilde{g%
}(e_{1},e_{2})=\tilde{g}(e_{1},e_{3})=\tilde{g}(e_{2},e_{3})=0.$ Then $(%
\tilde{\varphi},\xi ,\eta ,\tilde{g})$ is a paracontact metric structure on $%
M$. Using \ Lemma \ref{K>-1}, we conclude that $M$ is a generalized $(\tilde{%
\kappa}$,$\tilde{\mu}$) paracontact metric manifold with $\tilde{\kappa}%
=-1+z^{2},\tilde{\mu}=2(1-z).$
\end{example}

\begin{example}
\label{ex2}Consider the $3$-dimensional manifold 
\begin{equation*}
M=\left\{ (x,y,z)\in 
\mathbb{R}
^{3}\mid 2y-z\neq 0,\right\} ,
\end{equation*}%
where $(x,y,z)$ are the cartesian coordinates in $%
\mathbb{R}
^{3}$. We define three vector fields on $M$ as 
\begin{equation*}
e_{1}=(-2y+z)\frac{\partial }{\partial x}+(x-2y-z)\frac{\partial }{\partial y%
}+\frac{\partial }{\partial z},\text{ \ \ }e_{2}=\frac{\partial }{\partial y}%
,\text{ \ \ }e_{3}=\xi =\frac{\partial }{\partial x}.
\end{equation*}
The pseudo-Riemannian metric $\tilde{g}$, and the $(1,1)$-tensor field $%
\tilde{\varphi}$ given by 
\begin{equation*}
\tilde{g}=\left( 
\begin{array}{ccc}
1 & 0 & (2y-z)/2 \\ 
0 & 0 & \frac{1}{2} \\ 
(2y-z)/2 & \frac{1}{2} & (-2y+z)^{2}-2(x-2y-z)%
\end{array}%
\right) ,\text{ \ \ }\tilde{\varphi}\text{\ }=\left( 
\begin{array}{ccc}
0 & 0 & -2y+z \\ 
0 & -1 & x-2y-z \\ 
0 & 0 & 1%
\end{array}%
\right) .
\end{equation*}%
So we easily obtain $\tilde{g}(e_{1},e_{1})=\tilde{g}(e_{2},e_{2})=\tilde{g}%
(e_{1},e_{3})=\tilde{g}(e_{2},e_{3})=0$ and $\tilde{g}(e_{1},e_{2})=\tilde{g}%
(e_{3},e_{3})=1$. Moreover we have $\eta =dx+(2y-z)dz$\ and\ $\tilde{h}$\ $%
=\left( 
\begin{array}{ccc}
0 & 0 & 0 \\ 
0 & 0 & 1 \\ 
0 & 0 & 0%
\end{array}%
\right) $ with respect to the basis $\frac{\partial }{\partial x},\frac{%
\partial }{\partial y},\frac{\partial }{\partial z}$. By direct calculations
we get%
\begin{equation*}
\tilde{R}(X,Y)\xi =-(\eta (Y)X-\eta (X)Y)+2(\eta (Y)\tilde{h}X-\eta (X)%
\tilde{h}Y).
\end{equation*}%
Finally we deduce that $M$ is a $(-1,2,0)$-paracontact metric manifold.
\end{example}

\begin{remark}
\label{rem2}To our knowledge, the above example is the first numerical
example \ satisfying $\tilde{\kappa}=-1$ and $\tilde{h}\neq 0$ in $%
\mathbb{R}
^{3}$.
\end{remark}

\begin{example}
\label{ex3}In \cite{KMP} Koufogiorgos et al. construct following example.

Consider $3$-dimensional manifold 
\begin{equation*}
M=\left\{ (x,y,z)\in 
\mathbb{R}
^{3}\mid 2x+e^{y+z}>0,~\ y\neq z\right\}
\end{equation*}

and the vector fields $e_{1}=\frac{\partial }{\partial x},$%
\begin{eqnarray*}
e_{2} &=&\left( -\left( \frac{y^{2}+z^{2}}{2}\right) (2x+e^{y+z})^{\frac{1}{2%
}}\right) \frac{\partial }{\partial x}+\left( \frac{z(2x+e^{y+z})^{\frac{1}{2%
}}}{y-z}+\frac{(2x+e^{y+z})^{-\frac{1}{2}}}{y-z}\right) \frac{\partial }{%
\partial y} \\
&&\ +\left( \frac{y(2x+e^{y+z})^{\frac{1}{2}}}{z-y}+\frac{(2x+e^{y+z})^{-%
\frac{1}{2}}}{z-y}\right) \frac{\partial }{\partial z},
\end{eqnarray*}%
\begin{eqnarray*}
e_{3} &=&\left( \left( \frac{y^{2}+z^{2}}{2}\right) (2x+e^{y+z})^{\frac{1}{2}%
}\right) \frac{\partial }{\partial x}+\left( \frac{z(2x+e^{y+z})^{\frac{1}{2}%
}}{z-y}+\frac{(2x+e^{y+z})^{-\frac{1}{2}}}{y-z}\right) \frac{\partial }{%
\partial y} \\
&&+\left( \frac{y(2x+e^{y+z})^{\frac{1}{2}}}{y-z}+\frac{(2x+e^{y+z})^{-\frac{%
1}{2}}}{z-y}\right) \frac{\partial }{\partial z}.
\end{eqnarray*}%
Let $\eta $ be the $1$-form dual to $e_{1}$. The contact Riemannian
structure is defined as follows 
\begin{equation*}
\xi =e_{1},\varphi e_{1}=0,\varphi e_{2}=e_{3}\text{and }\varphi
e_{3}=-e_{2},
\end{equation*}%
\begin{equation*}
g(e_{i},e_{j})=\delta _{ij}\text{ }\ \text{for any }i, j\in \{1,2,3\}.\text{
\ \ \ \ \ \ \ \ \ \ }
\end{equation*}%
Thus \ it can be deduced that $(M,\varphi ,\xi ,\eta ,g)$ \ is a $(\kappa
,\mu ,\nu )$-contact metric manifold with $\kappa =1-\frac{1}{%
(2x+e^{y+z})^{2}},\mu =2$ and $\nu =\frac{-2}{(2x+e^{y+z})}.$

Next,using (\ref{PARA k<-1}) we can construct paracontact structure as
follows 
\begin{equation*}
\tilde{e}_{1}=\xi ,\tilde{e}_{2}=\frac{1}{\sqrt{2}}(e_{2}-e_{3}),\tilde{e}%
_{3}=\frac{1}{\sqrt{2}}(e_{2}+e_{3})\text{ such that }\tilde{\varphi}\tilde{e%
}_{1}=0,\tilde{\varphi}\tilde{e}_{2}=\tilde{e}_{3}\text{ and }\tilde{\varphi}%
\tilde{e}_{3}=\tilde{e}_{2}
\end{equation*}%
\begin{eqnarray*}
\tilde{g}(\tilde{e}_{1},\tilde{e}_{1}) &=&1,\text{ \ }\tilde{g}(\tilde{e}%
_{2},\tilde{e}_{2})=-1,\text{ \ }\tilde{g}(\tilde{e}_{3},\tilde{e}_{3})=1%
\text{ and} \\
\tilde{g}(\tilde{e}_{1},\tilde{e}_{2}) &=&\tilde{g}(\tilde{e}_{1},\tilde{e}%
_{3})=\tilde{g}(\tilde{e}_{2},\tilde{e}_{3})=0.\text{ \ \ \ \ \ \ \ \ \ \ }
\end{eqnarray*}%
Moreover, the matrix form of $\tilde{h}$ is given by 
\begin{equation*}
\tilde{h}=\left( 
\begin{array}{ccc}
0 & -\lambda & 0 \\ 
\lambda & 0 & 0 \\ 
0 & 0 & 0%
\end{array}%
\right) ,
\end{equation*}%
where $\lambda =\sqrt{1-\kappa }.$ After calculations, we finally deduce
that $(M,\tilde{\varphi},\xi ,\eta ,\tilde{g})$ is a $(\tilde{\kappa},\tilde{%
\mu},\tilde{\nu})$-paracontact metric manifold $\tilde{\kappa}=\kappa -2$, $%
\tilde{\mu}=2$ and $\tilde{\nu}=-\nu $.
\end{example}

\begin{remark}
\label{rema1}In the last example $\nu $ is a non-constant smooth function.
\end{remark}

\begin{remark}
\label{rema2}Choosing $\nu $ is a constant function, we can construct a
family of $(\tilde{\kappa}<-1,\tilde{\mu}=2,\tilde{\nu}=-\nu )$-paracontact
metric manifolds.
\end{remark}

We will finish this section by the following theorem.

\begin{theorem}
\label{TANGENT SPHERE}\textit{Let }$(M,\tilde{\varphi},\xi ,\eta ,\tilde{g})$%
\textit{\ be a }$3$\textit{-dimensional\textit{\ paracontact }}$(\tilde{%
\kappa},\tilde{\mu},\tilde{\nu})$\textit{\textit{-manifold }. If \ the
characteristic vector field }$\xi :(M,\tilde{g})\rightarrow (T_{1}M,\tilde{g}%
^{s})$ is harmonic map\textit{\ then paracontact }$(\tilde{\kappa},\tilde{\mu%
},\tilde{\nu})$\textit{-manifold is paracontact }$(\tilde{\kappa},\tilde{\mu}%
)$-manifold, i.e. $\tilde{\nu}=0$.
\end{theorem}

\begin{proof}
Using the definition of paracontact $(\tilde{\kappa},\tilde{\mu},\tilde{\nu}%
) $-manifold and properties of curvature tensor one has%
\begin{equation}
\tilde{R}(\xi ,W)X=\tilde{\kappa}(\tilde{g}(X,W)\xi -\eta (X)W)+\tilde{\mu}(%
\tilde{g}(\tilde{h}X,W)\xi -\eta (X)\tilde{h}W)+\tilde{\nu}(\tilde{g}(\tilde{%
\varphi}\tilde{h}X,W)\xi -\eta (X)\tilde{\varphi}\tilde{h}W).  \label{CEM}
\end{equation}%
Since the characteristic vector field $\xi $ is harmonic vector field for
the paracontact $(\tilde{\kappa},\tilde{\mu},\tilde{\nu})$-manifold it is
enough to calculate (\ref{TM}).

\textit{Case 1: }Assume that\textit{\ }$\tilde{\kappa}>-1.$

Using (\ref{A1}), (\ref{irem0}) and (\ref{CEM}) in (\ref{TM}), we get%
\begin{equation*}
tr[R(\nabla .\xi ,\xi ).]=(\tilde{\lambda}-1)\tilde{R}(\xi ,\tilde{\varphi}%
\tilde{e})\tilde{e}+(\tilde{\lambda}+1)\tilde{R}(\xi ,\tilde{e})\tilde{%
\varphi}\tilde{e}=2\tilde{\lambda^2}\tilde{\nu}\xi .
\end{equation*}%
So, we conclude that $tr[R(\nabla .\xi ,\xi ).]=0$ if and only if $\tilde{\nu%
}=0.$

\textit{Case 2: \ }We \ suppose that \textit{\ }$\tilde{\kappa}=-1.$ In this
case we know that $\tilde{\nu}$ vanishes.

By help of (\ref{cm1}), (\ref{cm2}), (\ref{irem3}) and (\ref{CEM}), we
obtain directly $tr[R(\nabla .\xi ,\xi ).]=0.$

\textit{Case 3: }Now we consider $\tilde{\kappa}<-1.$

Using (\ref{A3}), (\ref{iremmm0}) and (\ref{CEM}) in (\ref{TM}), we have%
\begin{equation*}
tr[R(\nabla .\xi ,\xi ).]=-2\tilde{\lambda}^{2}\tilde{\nu}\xi .
\end{equation*}%
Therefore, we deduce that $tr[R(\nabla .\xi ,\xi ).]=0$ if and only if $%
\tilde{\nu}=0.$

Thus, we complete the proof of the theorem.
\end{proof}

\section{An Application}

\label{fifth}

Now we will give some properties of $3$-dimensional contact metric manifolds.

Let $(M,\varphi ,\xi ,\eta ,g)$ be a contact metric $3$-manifold. Let%
\begin{eqnarray*}
U &=&\left\{ p\in M\mid h(p)\neq 0\right\} \subset M, \\
U_{0} &=&\left\{ p\in M\mid h(p)=0\text{, in a neighborhood of p}\right\}
\subset M.
\end{eqnarray*}

That $h$ is a smooth function on $M$ implies $U\cup U_{0}$ is an open and
dense subset of $M$, so any property satisfied in $U_{0}\cup U$ is also
satisfied in $M.$ For any point $p\in U\cup U_{0}$, there exists a local
orthonormal basis $\left\{ e,\varphi e,\xi \right\} $ of smooth eigenvectors
of h in a neighborhood of $p~$(this we call a $\varphi $-basis). On $U$, we
put $he=\lambda e,$ $h\varphi e=-\lambda \varphi e$, where $\lambda $ is a
nonvanishing smooth function assumed to be positive.

\begin{lemma}[\protect\cite{FP1}]
\label{CALVARUSO}(see also \cite{cal}) On the open set $U$ we have%
\begin{eqnarray}
\nabla _{\xi }e &=&a\varphi e,\text{ }\nabla _{e}e=b\varphi e,\text{ }\nabla
_{\varphi e}e=-c\varphi e+(\lambda -1)\xi ,  \label{3.1} \\
\nabla _{\xi }\varphi e &=&-ae,\nabla _{e}\varphi e=-be+(1+\lambda )\xi
,\nabla _{\varphi e}\varphi e=ce,  \label{3.2} \\
\nabla _{\xi }\xi &=&0,\text{\ }\nabla _{e}\xi =-(1+\lambda )\varphi e,\text{
}\nabla _{\varphi e}\xi =(1-\lambda )e,  \label{3.3} \\
\nabla _{\xi }h &=&-2ah\varphi +\xi (\lambda )s,  \label{3.4}
\end{eqnarray}%
\textit{where }$a$\textit{\ is a smooth function,}
\end{lemma}

\begin{eqnarray}
b &=&\frac{1}{2\lambda }(\varphi e(\lambda )+A)\text{ \ with \ }A=\eta
(Qe)=S(\xi ,e),  \label{3.5} \\
c &=&\frac{1}{2\lambda }(e(\lambda )+B)\text{ \ with \ }B=\eta (Q\varphi
e)=S(\xi ,\varphi e),  \label{3.6}
\end{eqnarray}%
\textit{and }$s$\textit{\ is the type }$(1,1)$\textit{\ tensor field defined
by }$s\xi =0$\textit{, }$se=e$\textit{\ and }$s\varphi e=-\varphi e.$

In \cite{BO}, Boeckx provided a local classification of non-Sasakian $%
(\kappa ,\mu )$-contact metric manifold respect to the number 
\begin{equation}
I_{M}=\frac{1-\frac{\mu }{2}}{\sqrt{1-\kappa }},  \label{Boeckx invariant}
\end{equation}%
which is an invariant of a $(\kappa ,\mu )$-contact metric manifold up to ${%
\mathcal{D}}_{\alpha }$-homothetic deformations.

For $(\kappa ,%
\mu
,\nu )$-contact metric manifolds, T. Koufogiorgos et al. \cite{KMP} proved
that the following relations hold 
\begin{eqnarray}
h^{2} &=&(\kappa -1)\varphi ^{2}\text{ for }\kappa \leq 1,  \label{HAR1} \\
\xi (\kappa ) &=&2\nu (\kappa -1),\text{ \ }\xi (\lambda )=\nu (\lambda ),
\label{HAR3} \\
\nabla _{\xi }h &=&\mu h\varphi +\nu h.  \label{HAR3a}
\end{eqnarray}%
Recently, in \cite{KEM}, the authors gave the following local classification
of a non-Sasakian $(\kappa ,%
\mu
,\nu )$-contact metric manifolds with $\xi (I_{M})=0$.

\begin{theorem}[\protect\cite{KEM}]
\label{ZETAMU}Let $(M,\varphi ,\xi ,\eta ,g)$ be a non-Sasakian $(\kappa
,\mu ,\nu =const.)$-contact metric manifold and $\xi (I_{M})=0$, where $\nu
=const.$ $\neq 0$. Then

$1)$ At any point of $M$, precisely one of the following relations is valid: 
$\mu =2(1+\sqrt{1-\kappa }),$ or $\mu =2(1-\sqrt{1-\kappa })$

$2)$ At any point $p\in M$ there exists a chart $(U,(x,y,z))$ with $p\in
U\subseteq M,$ such that

\ \ \ \ \ \ $i)$ the functions $\kappa ,\mu $ depend only on the variables $%
x $, $z.$

\ \ \ \ \ \ $ii)$ if $\mu =2(1+\sqrt{1-\kappa }),$ $($resp. $\mu =2(1-\sqrt{%
1-\kappa })),$ the tensor fields $\eta $, $\xi $, $\varphi $, $g$, $h$ are
given by the relations,%
\begin{equation*}
\xi =\frac{\partial }{\partial x},\text{ \ \ }\eta =dx-adz,
\end{equation*}%
\begin{equation*}
g=\left( 
\begin{array}{ccc}
1 & 0 & -a \\ 
0 & 1 & -b \\ 
-a & -b & 1+a^{2}+b^{2}%
\end{array}%
\right) \text{ \ \ \ \ }\left( \text{resp. \ \ }g=\left( 
\begin{array}{ccc}
1 & 0 & -a \\ 
0 & 1 & -b \\ 
-a & -b & 1+a^{2}+b^{2}%
\end{array}%
\right) \right) ,
\end{equation*}%
\begin{equation*}
\varphi =\left( 
\begin{array}{ccc}
0 & a & -ab \\ 
0 & b & -1-b^{2} \\ 
0 & 1 & -b%
\end{array}%
\right) \text{ \ \ \ \ }\left( \text{resp. \ \ }\varphi =\left( 
\begin{array}{ccc}
0 & -a & ab \\ 
0 & -b & 1+b^{2} \\ 
0 & -1 & b%
\end{array}%
\right) \right) ,
\end{equation*}%
\begin{equation*}
h=\left( 
\begin{array}{ccc}
0 & 0 & -a\lambda \\ 
0 & \lambda & -2\lambda b \\ 
0 & 0 & -\lambda%
\end{array}%
\right) \text{ \ \ \ \ \ }\left( \text{resp. \ \ }h=\left( 
\begin{array}{ccc}
0 & 0 & a\lambda \\ 
0 & -\lambda & 2\lambda b \\ 
0 & 0 & \lambda%
\end{array}%
\right) \right) ,
\end{equation*}%
with respect to the basis $\left( \frac{\partial }{\partial x},\frac{%
\partial }{\partial y},\frac{\partial }{\partial z}\right) ,$ where $%
a=2y+f(z)$ \ \ (resp.\ $a=-2y+f(z)$), $b=-\frac{y^{2}}{2}\nu -y\frac{f(z)}{2}%
\nu -\frac{y}{2}\frac{r^{^{\prime }}(z)}{r(z)}+\frac{2}{\nu }r(z)e^{\nu
x}+s(z)$ (resp. $b=\frac{y^{2}}{2}\nu -y\frac{f(z)}{2}\nu -\frac{y}{2}\frac{%
r^{^{\prime }}(z)}{r(z)}+\frac{2}{\nu }r(z)e^{\nu x}+s(z)$),\ $\lambda
=\lambda (x,z)=r(z)e^{\nu x}$ \ and $f(z)$, $r(z)$, $s(z)$ are arbitrary
smooth functions of $z$.
\end{theorem}

Now we are going to give a natural relation between non-Sasakian $(\kappa ,%
\mu
,\nu )$-contact metric manifolds with $\xi (I_{M})=0$ and $3$-dimensional
paracontact metric manifolds.

If Theorem \ref{motivation} is adapted for$\ $the $3$-dimensional
non-Sasakain $(\kappa ,\mu ,\nu )$-contact metric manifold and used same
procedure for proof then we have same result. Hence we can give following
theorem.

\begin{theorem}[\protect\cite{MOTE}]
\label{kmuvu}Let $(M,\varphi ,\xi ,\eta ,g)$ be a non-Sasakian $(\kappa ,\mu
,\nu )$-contact metric manifold. Then $M$ admits a canonical paracontact
metric structure $(\tilde{\varphi},\xi ,\eta ,\tilde{g})$ is given by 
\begin{equation}
\tilde{\varphi}:=\frac{1}{\sqrt{1-\kappa }}h,\ \ \tilde{g}:=\frac{1}{\sqrt{%
1-\kappa }}d\eta (\cdot ,h\cdot )+\eta \otimes \eta .  \label{PARA k<-1}
\end{equation}
\end{theorem}

After a long but straightforward calculation as in \cite{MOTE}, we get

\begin{proposition}
\label{RConnection}Let $(M,\varphi ,\xi ,\eta ,g)$ be a non-Sasakian $%
(\kappa ,\mu ,\nu )$-contact metric manifold. Then the Levi-Civita
connections $\nabla $ and $\tilde{\nabla}$ of $g$ and $\tilde{g}$ are
related as 
\begin{eqnarray}
\tilde{\nabla}_{X}Y &=&\nabla _{X}Y+\frac{1}{2(1-\kappa )}\varphi h(\nabla
_{X}\varphi h)Y-\frac{1}{\sqrt{1-\kappa }}\eta (Y)hX-\frac{1}{\sqrt{1-\kappa 
}}\eta (X)hY  \notag \\
&&-\frac{1}{2}\eta (Y)\varphi hX-\frac{(1-\mu )}{2}\eta (Y)\varphi X-\frac{%
\nu }{2}\eta (Y)\varphi ^{2}X  \notag \\
&&+(\frac{1}{2\sqrt{1-\kappa }}g(hX,Y)+\sqrt{1-\kappa }g(X,Y)-\sqrt{1-\kappa 
}\eta (X)\eta (Y)  \label{RELATED CONNECTION} \\
&&+\frac{(1-\mu )}{2\sqrt{1-\kappa }}g(hX,Y)-g(X,\varphi Y)+X(\eta (Y))-\eta
(\nabla _{X}Y))\xi  \notag \\
&&-\frac{1}{2}(1-\kappa )(X(\frac{1}{\sqrt{1-\kappa }})\varphi ^{2}Y+Y(\frac{%
1}{\sqrt{1-\kappa }})\varphi ^{2}X+  \notag \\
&&+\frac{1}{(1-\kappa )}g(X,\varphi hY)\varphi h \text{grad}(\frac{1}{\sqrt{%
1-\kappa }})),  \notag
\end{eqnarray}%
for any $X,Y\in \Gamma (TM)$.
\end{proposition}

Now we will give a relation between $\tilde{h}$ and $h$ by following lemma.

\begin{lemma}
\label{HB}Let $(M,\varphi ,\xi ,\eta ,g)$ be a non-Sasakian $(\kappa ,\mu
,\nu )$-contact metric manifold and let $(\tilde{\varphi},\xi ,\eta ,\tilde{g%
})$ be the canonical paracontact metric structure induced on $M$, according
to Theorem \ref{kmuvu}. Then we have%
\begin{equation}
\tilde{h}=\frac{1}{2\sqrt{1-\kappa }}((2-\mu )\varphi \circ h+2(1-\kappa
)\varphi ).  \label{H BAR2}
\end{equation}
\end{lemma}

\begin{proof}
By help the equations (\ref{HAR3}), (\ref{PARA k<-1}) and the definitons of $%
\tilde{h}$ and $h$ we have%
\begin{eqnarray}
2\tilde{h} &=&L_{\xi }\tilde{\varphi}=L_{\xi }(\frac{1}{\sqrt{1-\kappa }}h) 
\notag \\
&=&-\frac{\nu }{\sqrt{1-\kappa }}h+\frac{1}{2\sqrt{1-\kappa }}L_{\xi
}(L_{\xi }\varphi ).  \label{HAN1}
\end{eqnarray}%
Using the identities $\nabla \xi =-\varphi -\varphi h,\nabla _{\xi }\varphi
=0$ and $\varphi ^{2}h=-h,$ and following same procedure cf. [\cite{MOTE}
pg. 270], we get%
\begin{equation*}
2\tilde{h}=-\frac{\nu }{\sqrt{1-\kappa }}h+\frac{1}{2\sqrt{1-\kappa }}%
(2\nabla _{\xi }h+4h^{2}\varphi -4h\varphi ).
\end{equation*}%
By using (\ref{HAR1}) and (\ref{HAR3a}), we obtain the claimed relation.
\end{proof}

Let $(M,\varphi ,\xi ,\eta ,g)$ be a non-Sasakian $(\kappa ,\mu ,\nu )$%
-contact metric manifold. Choosing a local orthonormal $\varphi $-basis $%
\{e,\varphi e,\xi \}$ on $(M,\varphi ,\xi ,\eta ,g)$ and using Proposition %
\ref{RConnection}, Lemma \ref{HB} and Lemma \ref{CALVARUSO}, one can give
following proposition.

\begin{proposition}
\label{HB2}Let $(M,\varphi ,\xi ,\eta ,g)$ be a non-Sasakian $(\kappa ,\mu
,\nu )$-contact metric manifold and let $(\tilde{\varphi},\xi ,\eta ,\tilde{g%
})$ be the canonical paracontact metric structure induced on $M$, according
to Theorem \ref{kmuvu}. Then we have%
\begin{eqnarray*}
i)\text{ }\tilde{\nabla}_{e}e &=&-\frac{1}{2\lambda }e(\lambda )e+(\lambda
+1-\frac{\mu }{2})\xi \text{ \ \ }ii)\text{ }\tilde{\nabla}_{e}\varphi e=%
\frac{1}{2\lambda }e(\lambda )\varphi e+\xi ,\text{ \ }iii)\text{ }\tilde{%
\nabla}_{e}\xi =-e+(\frac{\mu }{2}-1-\lambda )\varphi e, \\
iv)\text{ }\tilde{\nabla}_{\varphi e}e &=&\frac{1}{2\lambda }(\varphi
e)(\lambda )e-\xi ,\text{ \ }v)\text{ }\tilde{\nabla}_{\varphi e}\varphi e=-%
\frac{1}{2\lambda }(\varphi e)(\lambda )\varphi e+(\lambda -1+\frac{\mu }{2}%
)\xi ,\text{ \ } \\
vi)\text{ }\tilde{\nabla}_{\varphi e}\xi &=&-e-(\lambda -1+\frac{\mu }{2}%
)\varphi e,\text{ \ \ } \\
vii)\text{ }\tilde{\nabla}_{\xi }e &=&-e,\text{ \ }viii)\text{ }\tilde{\nabla%
}_{\xi }\varphi e=\varphi e,\text{ \ \ } \\
ix)\text{ }[e,\xi ] &=&(\frac{\mu }{2}-1-\lambda )\varphi e,\text{\ }x)\text{
}[\varphi e,\xi ]=-e-(\lambda +\frac{\mu }{2})\varphi e,\text{ \ }xi)\text{ }%
[e,\varphi e]=-\frac{1}{2\lambda }(\varphi e)(\lambda )e+\frac{1}{2\lambda }%
e(\lambda )\varphi e+2\xi \text{.}
\end{eqnarray*}

Moreover, $\tilde{g}(e,e)=\tilde{g}(\varphi e,\varphi e)=\tilde{g}(\varphi
e,\xi )=0$ and $\tilde{g}(e,\varphi e)=\tilde{g}(\xi ,\xi )=1$.
\end{proposition}

We assume that $(M,\varphi ,\xi ,\eta ,g)$ be a non-Sasakian $(\kappa ,\mu
,\nu )$-contact metric manifold, $(\tilde{\varphi},\xi ,\eta ,\tilde{g})$ be
the canonical paracontact metric structure induced on $M$, according to
Theorem \ref{kmuvu}. Let \ $\{e,\varphi e,\xi \}$ be an orthonormal $\varphi 
$-basis in a neighborhood of $p\in M.$ Then one can always construct an
orthonormal $\tilde{\varphi}$-basis $\{\tilde{e}_{1},\tilde{\varphi}\tilde{e}%
_{1}=\tilde{e}_{2},\xi \}$, for instance $\tilde{e}_{1}=(e-\varphi e)/\sqrt{2%
},$ $\tilde{\varphi}\tilde{e}_{1}=(e+\varphi e)/\sqrt{2},$ such that $\tilde{%
g}(\tilde{e}_{1},\tilde{e}_{1})=-1,$ $\tilde{g}(\tilde{\varphi}\tilde{e}_{1},%
\tilde{\varphi}\tilde{e}_{1})=1,$ $\tilde{g}(\xi ,\xi )=1$. Moreover, from
Lemma \ref{HB}, the matrix form of $\tilde{h}$ is given by 
\begin{equation}
\tilde{h}=\left( 
\begin{array}{ccc}
-1+\frac{\mu }{2} & -\lambda & 0 \\ 
\lambda & 1-\frac{\mu }{2} & 0 \\ 
0 & 0 & 0%
\end{array}%
\right)  \label{HSAPKA}
\end{equation}%
with respect to local orthonormal basis $\{\tilde{e}_{1},\tilde{\varphi}%
\tilde{e}_{1},\xi \}$. By using Proposition \ref{HB2} \ we have following
proposition.

\begin{proposition}
\label{HB3}Let $(M,\varphi ,\xi ,\eta ,g)$ be a non-Sasakian $(\kappa ,\mu
,\nu )$-contact metric manifold and let $(\tilde{\varphi},\xi ,\eta ,\tilde{g%
})$ be the canonical paracontact metric structure induced on $M$, according
to Theorem \ref{kmuvu}. Then we have%
\begin{eqnarray}
i)\text{ }\tilde{\nabla}_{\tilde{e}_{1}}\tilde{e}_{1} &=&-\frac{1}{2\lambda }%
(\tilde{\varphi}\tilde{e}_{1})(\lambda )\tilde{\varphi}\tilde{e}_{1}+\lambda
\xi ,\text{ \ \ }ii)\text{ }\tilde{\nabla}_{\tilde{e}_{1}}\tilde{\varphi}%
\tilde{e}_{1}=-\frac{1}{2\lambda }(\tilde{\varphi}\tilde{e}_{1})(\lambda )%
\tilde{e}_{1}+(2-\frac{\mu }{2})\xi ,\text{ }  \notag \\
iii)\text{ }\tilde{\nabla}_{\tilde{e}_{1}}\xi &=&\lambda \tilde{e}_{1}+(%
\frac{\mu }{2}-2)\tilde{\varphi}\tilde{e}_{1},\text{ \ }iv)\text{ }\tilde{%
\nabla}_{\tilde{\varphi}\tilde{e}_{1}}\tilde{e}_{1}=-\frac{1}{2\lambda }%
\tilde{e}_{1}(\lambda )\tilde{\varphi}\tilde{e}_{1}-\frac{\mu }{2}\xi , 
\notag \\
v)\text{ }\tilde{\nabla}_{\tilde{\varphi}\tilde{e}_{1}}\tilde{\varphi}\tilde{%
e}_{1} &=&-\frac{1}{2\lambda }\tilde{e}_{1}(\lambda )\tilde{e}_{1}+\lambda
\xi ,\text{ \ \ }vi)\text{ }\tilde{\nabla}_{\tilde{\varphi}\tilde{e}_{1}}\xi
=-\frac{\mu }{2}\tilde{e}_{1}-\lambda \tilde{\varphi}\tilde{e}_{1},  \notag
\\
vii)\text{ }\tilde{\nabla}_{\xi }\tilde{e}_{1} &=&-\tilde{\varphi}\tilde{e}%
_{1},\text{ \ \ }viii)\text{ }\tilde{\nabla}_{\xi }\tilde{\varphi}\tilde{e}%
_{1}=-\tilde{e}_{1},\text{ \ }  \notag \\
ix)\text{ }[\tilde{e}_{1},\xi ] &=&\lambda \tilde{e}_{1}+(\frac{\mu }{2}-1)%
\tilde{\varphi}\tilde{e}_{1}\text{ \ }x)\text{ }[\tilde{\varphi}\tilde{e}%
_{1},\xi ]=(1-\frac{\mu }{2})\tilde{e}_{1}-\lambda \tilde{\varphi}\tilde{e}%
_{1},\text{ }  \notag \\
xi)\text{ }[\tilde{e}_{1},\tilde{\varphi}\tilde{e}_{1}] &=&-\frac{1}{%
2\lambda }(\tilde{\varphi}\tilde{e}_{1})(\lambda )\tilde{e}_{1}+\frac{1}{%
2\lambda }\tilde{e}_{1}(\lambda )\tilde{\varphi}\tilde{e}_{1}+2\xi .
\label{IM00}
\end{eqnarray}
\end{proposition}

\begin{proposition}
\label{HB4}Let $(M,\varphi ,\xi ,\eta ,g)$ be a non-Sasakian $(\kappa ,\mu
,\nu =const.)$-contact metric manifold with $\xi (I_{M})=0$ and suppose that 
$(\tilde{\varphi},\xi ,\eta ,\tilde{g})$ be the canonical paracontact metric
structure induced on $M.$ Then the following equation holds%
\begin{equation}
\tilde{\nabla}_{\xi }\tilde{h}=2\tilde{h}\tilde{\varphi}+\nu \tilde{h}\text{.%
}  \label{IM1}
\end{equation}
\end{proposition}

\begin{proof}
Taking $\xi (\mu )=\nu (\mu -2)$ and $\xi (\lambda )=\nu \lambda $ into
account and using the relations (\ref{HSAPKA}), (\ref{IM00}) we get
requested relation.
\end{proof}

We suppose that $(M,\varphi ,\xi ,\eta ,g)$ be a non-Sasakian $(\kappa ,\mu
,\nu )$-contact metric manifold with $\xi (I_{M})=0$. Let $(\tilde{\varphi}%
,\xi ,\eta ,\tilde{g})$ be the canonical paracontact metric structure
induced on $M$, according to Theorem \ref{kmuvu}. By (\ref{CURVATURE 4}) and
(\ref{IM00}), and after very long computations we obtain that%
\begin{equation}
\tilde{R}(\tilde{e}_{1},\tilde{\varphi}\tilde{e}_{1})\xi =(\frac{1}{\lambda }%
(\frac{\mu }{2}-1)\tilde{e}_{1}(\lambda )-\frac{1}{2}\tilde{e}_{1}(\mu ))%
\tilde{e}_{1}+(\frac{1}{\lambda }(\frac{\mu }{2}-1)(\tilde{\varphi}\tilde{e}%
_{1})(\lambda )-\frac{1}{2}(\tilde{\varphi}\tilde{e}_{1})(\mu ))\tilde{%
\varphi}\tilde{e}_{1}.  \label{RM1}
\end{equation}%
By using Theorem \ref{ZETAMU} and $\tilde{e}_{1}=(e-\varphi e)/\sqrt{2},%
\tilde{\varphi}\tilde{e}_{1}=(e+\varphi e)/\sqrt{2}$ in (\ref{RM1}) one can
check that in fact $\tilde{R}(\tilde{e}_{1},\tilde{\varphi}\tilde{e}_{1})\xi
=0$. Now if we use (\ref{THREE DIM CURVATURE}), we have%
\begin{equation}
\tilde{R}(\tilde{e}_{1},\tilde{\varphi}\tilde{e}_{1})\xi =-\tilde{\sigma}(%
\tilde{e}_{1})\tilde{\varphi}\tilde{e}_{1}+\tilde{\sigma}(\tilde{\varphi}%
\tilde{e}_{1})\tilde{e}_{1}.  \label{RM2}
\end{equation}%
Comparing $\tilde{R}(\tilde{e}_{1},\tilde{\varphi}\tilde{e}_{1})\xi =0$ with
(\ref{RM2}), we get%
\begin{equation}
\tilde{\sigma}(\tilde{e}_{1})=\tilde{\sigma}(\tilde{\varphi}\tilde{e}_{1})=0%
\text{.}  \label{RM3}
\end{equation}%
So $\xi $ is an eigenvector of Ricci operator $\tilde{Q}.$

\begin{remark}
\label{remark3}From (\ref{NAMBLASTAR}) and (\ref{IM00}) we have%
\begin{eqnarray}
\tilde{\nabla}^{\ast }\tilde{\nabla}\xi &=&-\tilde{\nabla}_{\tilde{e}_{1}}%
\tilde{\nabla}_{\tilde{e}_{1}}\xi +\tilde{\nabla}_{\tilde{\nabla}_{\tilde{e}%
_{1}}\tilde{e}_{1}}\xi +\tilde{\nabla}_{\tilde{\varphi}\tilde{e}_{1}}\tilde{%
\nabla}_{\tilde{\varphi}\tilde{e}_{1}}\xi -\nabla _{\nabla _{\tilde{\varphi}%
\tilde{e}_{1}}\tilde{\varphi}\tilde{e}_{1}}\xi  \notag \\
&=&(\frac{1}{\lambda }(\frac{\mu }{2}-1)(\tilde{\varphi}\tilde{e}%
_{1})(\lambda )-\frac{1}{2}(\tilde{\varphi}\tilde{e}_{1})(\mu ))\tilde{e}_{1}
\notag \\
&&+(\frac{1}{\lambda }(\frac{\mu }{2}-1)\tilde{e}_{1}(\lambda )-\frac{1}{2}%
\tilde{e}_{1}(\mu ))\tilde{\varphi}\tilde{e}_{1}  \label{IM2} \\
&&+((2-\frac{\mu }{2})^{2}+\frac{\mu ^{2}}{4}-2\lambda ^{2})\xi \text{.} 
\notag
\end{eqnarray}%
Then by using again Theorem \ref{ZETAMU} and $\tilde{e}_{1}=(e-\varphi e)/%
\sqrt{2},\tilde{\varphi}\tilde{e}_{1}=(e+\varphi e)/\sqrt{2}$ in (\ref{IM2})
\ and after long computations one can prove that the last relation reduces
to 
\begin{equation}
\tilde{\nabla}^{\ast }\tilde{\nabla}\xi =2\xi .  \label{IM3}
\end{equation}%
By virtue of (\ref{IM3}), we see immediately that $\xi $ is harmonic vector
field on $(M,\tilde{\varphi},\xi ,\eta ,\tilde{g}).$
\end{remark}

So we can give following theorem:

\begin{theorem}
\label{HPARACONTACT}Let $(M,\varphi ,\xi ,\eta ,g)$ be a non-Sasakian $%
(\kappa ,\mu ,\nu =const)$-contact metric manifold with $\xi (I_{M})=0$ and
let $(\tilde{\varphi},\xi ,\eta ,\tilde{g})$ be the canonical paracontact
metric structure induced on $M$, according to Theorem \ref{kmuvu}. Then the
characteristic vector field $\xi $ is an eigenvector of the Ricci operator $%
. $
\end{theorem}

By using previous theorem and same procedure as in Case 3 of the proof of
Theorem \ref{k mu vu}, we obtain following theorem.

\begin{theorem}
\label{HPARACONTACT K MU VU}Let $(M,\varphi ,\xi ,\eta ,g)$ be a
non-Sasakian $(\kappa ,\mu ,\nu =const)$-contact metric manifold with $\xi
(I_{M})=0$ and let $(\tilde{\varphi},\xi ,\eta ,\tilde{g})$ be the canonical
paracontact metric structure induced on $M$, according to Theorem \ref{kmuvu}%
. Then the curvature tensor field of the Levi Civita connection of $(M,%
\tilde{g})$ verifies the following relation%
\begin{equation*}
\tilde{R}(X,Y)\xi =(\kappa -2)(\eta (Y)X-\eta (X)Y)+2(\eta (Y)\tilde{h}%
X-\eta (X)\tilde{h}Y)-\nu (\eta (Y)\tilde{\varphi}\tilde{h}X-\eta (X)\tilde{%
\varphi}\tilde{h}Y).
\end{equation*}
\end{theorem}

\textbf{Acknowledgement: }The authors wish to thank to G. Calvaruso, D.
Perrone and Q. Pan for useful comments on the manuscript.

\end{document}